\documentclass[12pt,a4paper]{article}
\usepackage{amsmath}
\usepackage{amsfonts}
\usepackage{amssymb}
\usepackage{amsthm}
\usepackage{mathtools}
\usepackage{mathabx}
\usepackage{xcolor}
\usepackage{hyperref}
 \usepackage[bottom=4cm]{geometry}

\usepackage[latin1]{inputenc} 
\usepackage[french,english]{babel}  
\usepackage{paralist}
\usepackage{natbib}

\usepackage{bbm}
\usepackage{pifont}

\newtheorem{Assumption}{{\bf Assumption}}[section]
\newtheorem{lemma}{Lemma}[section]
\newtheorem{Remark}{{\bf Remark}}[section]
\newtheorem{definition}{{\bf Definition}}[section]
\newtheorem{theorem}{Theorem}[section]
\newtheorem{proposition}{Proposition}[section]

\newtheorem{remark}{Remark}[section]

\makeatletter
\renewenvironment{proof}[1][\proofname]{%
  \par\pushQED{\qed}\normalfont%
  \topsep6\p@\@plus6\p@\relax
  \trivlist\item[\hskip\labelsep\bfseries#1\@addpunct{.}]%
  \ignorespaces
}{%
  \popQED\endtrivlist\@endpefalse
}
\makeatother

\newcommand{\R}{\mathbb{R}}
\newcommand{\N}{\mathbb{N}}
\newcommand{\I}{\mathbb{I}}
\newcommand{\cf}{\mathcal{F}}

\newcommand{\F}{\mathbb{F}}

\long\def\symbolfootnote[#1]#2{\begingroup\def\thefootnote{\fnsymbol{footnote}}
\footnote[#1]{#2}\endgroup}

\def\esssup{\text{ess sup}} 
\DeclareMathOperator*{\essinf}{ess\,inf}

\begin{document}


\title{Optimal stopping: \\ Bermudan strategies meet  non-linear evaluations  
}

\author{ 
 Miryana Grigorova
 \and
 Marie-Claire Quenez 
 \and
 Peng Yuan  
}
\date{13 January 2023}
\maketitle




\textit{Abstract: }

We address an optimal stopping problem 
over the set of Bermudan-type strategies $\Theta$ (which we understand in a more general sense than the stopping strategies for Bermudan options in finance) and with  non-linear operators (non-linear evaluations) assessing the rewards, under general assumptions on the non-linear operators $\rho$.  We provide a characterization of the value family $V$ in terms of what we call the $(\Theta,\rho)$-Snell envelope of the  pay-off family. We establish a Dynamic Programming Principle. We provide an optimality criterion in terms of a $(\Theta,\rho)$-martingale property of $V$ on a stochastic interval. We investigate the $(\Theta,\rho)$-martingale structure and we show that  the "first time" when the value family coincides with the pay-off family is optimal. The reasoning simplifies in the case where there is a finite number $n$ of pre-described stopping times, where $n$ does not depend on the  scenario $\omega$. We  provide examples of non-linear operators entering our framework.  \\


%
%
\section{Introduction}

In the recent years, optimal stopping problems   with non-linear evaluations have gained  an increasing interest in the financial mathematics literature and in the stochastic control literature.  \\

In the \emph{linear} case, a classical reference are the notes by \cite{ElKaroui_notes}. A presentation based on families of random variables indexed by stopping times can be found in \cite{Quenez-Kobylanski}. 
In discrete time, a \emph{non-linear} optimal stopping  with  dynamic monetary utilities was studied in \cite{Schoe}, and with 
 $g$-evaluations (induced by Backward SDEs with Lipschitz driver $g$) -- in \cite{Grigorova-2}.
In continuous time, a  non-linear optimal stopping with dynamic convex  risk measures was studied in \cite{{Bayraktar-2}}; with the so-called  $F$-expectations --in \cite{Bayraktar} and \cite{{Bayraktar-3}};   with g-evaluations--in e.g.,  \cite{Quenez-Sulem},  \cite{Grigorova-1},  \cite{Grigorova-3},  \cite{Klimsiak}; with a focus on applications  to American options in complete or incomplete non-linear financial markets -- in \cite{Rutkowski}, \cite{Grigorova_Am}; with suprema of linear/affine operators over sets of measures- in, e.g. \cite{Ekren}, \cite{Nutz}.\\

In the present paper, we address an optimal stopping problem 
with  \textit{Bermudan-type} strategies and with general  non-linear operators (non-linear evaluations) assessing the rewards. \\

Our purpose is two-fold: 
\begin{itemize}
\item[1.] We consider a modelling framework which is in-between the discrete-time and the continuous-time framework, by focusing on what we call in this paper the Bermudan-type stopping strategies    
   \footnote{Note that this notion is more general than the typical notion of Bermudan strategy appearing in the context of Bermudan options in mathematical finance, where the agent/buyer of the option is allowed to stop only at a pre-described \textbf{finite} number of \textbf{deterministic} times $\{0=t_0\leq t_1\leq \cdots\leq t_n=T\},$ where $n$ is a fixed  finite number (\textbf{independent} of the scenario $\omega$).  }.  
    
   -- In the discrete-time framework with finite terminal horizon $T>0$,  the agent is allowed to stop at a\textbf{ finite number} only of pre-described \textbf{determinstic} times, and  gain/loss processes are indexed by these pre-described determinsitic times.  If we denote by $\{0=t_0\leq t_1\leq ...\leq t_n=T\}$  the pre-defined finite deterministic grid of $n+1$ time points, the stopping strategies of the agent are  of the form $\tau=\sum_{k=0}^{n} t_k {\bf 1}_{A_k},$ where $(A_k)_{k\in\{0,1,...n\}}$ is a partition, such that  $A_k$ is $\cf_{t_k}$-measurable, for each $k\in\{0,1,\ldots,n\}$. Thus, for  almost each scenario $\omega$, the agent  is allowed to stop only at a  finite number of times (provided they do that in a non-anticipative way), where both the number of time instants  (here, $n+1$) and the time instants themselves (here, the $t_k$'s), are the same, whatever the scenario $\omega$.    
   
  --  In the continuous-time framework (with finite horizon $T>0$), the agent is allowed to stop continuously at any time instant $t\in[0,T]$, and the  gain/loss processes are indexed by $t\in[0,T]$. The set of the agent's stopping strategies is the set of all stopping times valued in $[0,T]$.  Thus, for almost each scenario $\omega$, the agent can stop at any time instant (provided they do that in a non-anticipative way).
  
  -- In the \textit{intermediate}  modelling framework of the current paper (with finite terminal horizon $T>0$), in (almost) every scenario $\omega$, the agent is allowed   to stop  at a  finite number of times or infinite countable number of times (provided they do that in a non-anticipative way), where both the number of time instants and the time instants themselves are allowed to depend on the scenario $\omega$.  More specifically, we are given a non-decreasing  sequence of stopping times $(\theta_k)_{k\in\mathbb{N}}$ such that $\lim_{k\to\infty} \theta_k=T. $ This countable stopping grid being given, the agent's  stopping strategies $\tau$  are thus of the form $\tau=\sum_{k=0}^{+\infty} \theta_k {\bf 1}_{A_k}+T{\bf 1}_{\bar A},$ where $\{(A_k)_{k\in\mathbb{N}},\bar A\}$ is a partition, such that   $A_k$ is $\cf_{\theta_k}$-measurable, for each $k\in\mathbb{N},$ and $\bar A$ is $\cf_T$-measurable. We call $\tau$ of this form a Bermudan stopping strategy, and we denote by $\Theta$ the set of Bermudan stopping strategies.  The gain/losses are then "naturally" defined via  families of random variables indexed by the stopping times $\tau$ of this  form.\\
[0.15cm]
This modelling framework is thus closer to the real-life situations where the number of possible decision points depends on the scenario/state of nature, and so do  the decision times themselves, but where the agents do not necessarily act continuously in time.  
   \item[2.] The second purpose of the paper is to allow for  gains/losses being assessed by general non-linear evaluations $\rho = (\rho_{S, \tau}[\cdot])$, while  imposing minimal assumptions  on the  non-linear  operators, under which the results hold. 
\end{itemize}
\noindent
We note also that,  in the above framework,  working with \textbf{families of random variables} $\phi=(\phi(\tau))$ indexed by Bermudan stopping times $\tau$,  allows for an exposition in which it is not necessary to invoke  any results from the theory of stochastic processes.  \\
After formulating  the non-linear optimal stopping with Bermudan-style strategies,  %
we provide a characterization of the value family  in terms of  a suitably defined non-linear Snell envelope. A dynamic programming principle is established under suitable assumptions on the non-linear evaluations. 
An optimality criterion is proven  and existence of optimal stopping times is investigated; it is shown in particular that the first "hitting time" is optimal.  Examples of  non-linear operators, well-known in financial mathematics   and in  stochastic control, entering our framework, are given, such as  the non-linear evaluations induced by Backward SDEs, the non-linear $F$-expectations introduced  by Bayraktar and Yao, as well as the dynamic concave utilities (from the risk measurement literature).    In the appendix, we consider the particular case of finite number of pre-described stopping times  $0=\theta_0,\theta_1,...,\theta_n=T$, where $n$ does not depend on the scenario $\omega$, and provide an explicit construction of the non-linear Snell envelope by backward induction, as well as a simpler proof of the optimality of the first hitting time.\\
[0.3cm]
The remainder of the paper is organized as follows:
In Subsection \ref{defin_sect}, we set the framework and the notation. In Subsection \ref{optimisation_problem}, we formulate the optimisation problem. In Subsection \ref{subsect2}, we characterize the value family of the problem in terms of the $(\Theta, \rho)$-Snell envelope family of the pay-off family. In Subsection \ref{SUBset4}, we show a Dynamic Programming Principle. In Subsection \ref{SUBset5}, we investigate the question of optimal stopping times: we provide some technical lemmas regarding the $(\Theta, \rho)$-martingale property, we provide an optimality criterion, as well as some useful consequences of the DPP, and we show that, under suitable assumptions,  ``the first time'' $\nu_{k}$ when the value family hits the pay-off family is optimal. Section \ref{SectioN3} is dedicated to three examples of non-linear operators from the literature, entering our framework. The Appendix is dedicated to the particular case where in almost each scenario $\omega \in \Omega$, there are $n+1$ pre-described opportunities for stopping, where $n$ does \textbf{not} depend on $\omega \in \Omega$. In this case, we have an explicit construction of the $(\Theta, \rho)$-Snell envelope by backward induction and a simplified proof of the optimality of $\nu_{k}$ (requiring ``less continuity'' on $\rho$).


\section{Optimal stopping with non-linear evaluations and Bermudan strategies}\label{sect_optimal}

\subsection{The framework}\label{defin_sect}

Let $T>0$ be a \textbf{fixed finite} terminal horizon.\\
Let $(\Omega,  \cf, P)$ be a (complete) probability space equipped with a right-continuous complete filtration $\F = \{\mathcal{F}_t \colon t\in[0,T]\}$. \\
In the sequel, equalities and inequalities between random variables are to be understood in the $P$-almost sure sense. Equalities between measurable sets are to be understood in the $P$-almost sure sense.\\
Let $\N$ be the set of natural numbers, including $0$.  Let $\N^*$ be the set of natural numbers, excluding $0$.
Let ($\theta_k)_{k\in\N}$ be a sequence of stopping times satisfying the following properties:
\begin{itemize}
\item[(a)] The sequence $(\theta_k)_{k\in\N}$  is non-decreasing, i.e. 
for all $k\in\N$, $\theta_k\leq \theta_{k+1}$, a.s. 
\item[(b)] $\lim_{k\to\infty}\uparrow \theta_k=T$ a.s. 
\end{itemize}
Moreover, we set $\theta_0=0$. \\
We note that the family of $\sigma$-algebras 
$({\cal F}_{\theta_k})_{k\in\N}$ is non- decreasing (as the sequence $(\theta_{k})$ is non-decreasing).
We denote by ${\Theta}$ the set of stopping times $\tau$ of the form 
\begin{equation}\label{form}
\tau= \sum_{k=0}^{+\infty} \theta_k {\bf 1}_{A_k}+T{\bf 1}_{\bar A},
\end{equation}
where $\{(A_{k})^{+\infty}_{k=0}, \bar{A}\}$ form a partition of $\Omega$ such that, for each $k\in\N$,  $A_{k} \in \mathcal{F}_{\theta_{k}}$, and $\bar{A} \in \mathcal{F}_{T}$.\\
The set ${\Theta}$ can also  be described as the set of stopping times $\tau$ such that for almost all $\omega \in \Omega$, either $\tau(\omega) = T$ or $\tau(\omega) = \theta_{k}(\omega)$, for some $k = k(\omega) \in \mathbbm{N}$.\\
Note that the set $\Theta$ is closed under concatenation, that is, for each $\tau$ $\in$ $\Theta$ and each 
$A \in {\cal F}_{\tau}$, the stopping time $\tau {\bf 1}_{A} + T {\bf 1}_{A^c}$ $\in$ $\Theta$. More generally, for each $\tau$ $\in$ $\Theta$, $\tau'$ $\in$ $\Theta$ and each 
$A \in {\cal F}_{\tau\wedge \tau'}$, the stopping time $\tau {\bf 1}_{A} + \tau' {\bf 1}_{A^c}$ is in  $\Theta$.  The set $\Theta$ is also closed under pairwise minimization (that is, for  each $\tau\in\Theta$ and $\tau'\in\Theta$, we have $\tau\wedge \tau'\in\Theta$) and under pairwise maximization (that is, for  each $\tau\in\Theta$ and $\tau'\in\Theta$, we have $\tau\vee \tau'\in\Theta$). Moreover,
for each non-decreasing (resp. non-increasing) sequence of stopping times 
$(\tau_n)_{n \in {\mathbb N}} \in \Theta^{\mathbbm{N}}$, we have $\lim_{n \rightarrow + \infty} \tau_n$  $\in$ $\Theta$.\\
We note also that all stopping times in $\Theta$ are bounded from above by $T$. 
\begin{Remark}\label{Rk_canonical} We have the following \emph{canonical}  writing of the sets in \eqref{form}: 
\begin{align*}
A_0&=\{\tau=\theta_0\};\\
A_{n+1}&=\{\tau=\theta_{n+1},\theta_{n+1}<T\}\backslash (A_n\cup...\cup A_0); \text { for all } n\in\N^*\\
\bar A&=(\cup_{k=0}^{+\infty} A_k)^c  
\end{align*}  
From this writing, we have: if $\omega\in A_{k+1}\cap \{\theta_k<T\}$, then $\omega\notin \{\tau=\theta_k\}.$ 
\end{Remark}
\noindent
For each $\tau \in \Theta$, we denote by $\Theta_{\tau}$ the set of stopping times $\nu \in 
\Theta$ such that $\nu \geq \tau$ a.s.\, The set $\Theta_{\tau}$ satisfies the same properties as 
the set $\Theta$.
We will refer to the set $\Theta$ as the set of \textbf{Bermudan stopping strategies},
and to the set $\Theta_\tau$ as the set of Bermudan stopping strategies, greater than or equal to $\tau$ (or the set of Bermudan stopping strategies from time $\tau$ perspective).

\begin{definition}\label{def.admi}
We say that a family 
$\phi=(\phi(\tau), \, \tau \in \Theta)$  is \emph{admissible} if it satisfies the following conditions 
\par
1. \quad for all
$\tau \in \Theta$, $\phi(\tau)$ is a real valued random variable, which is  $\mathcal{F}_\tau$-measurable. \par
 2. \quad  for all
$\tau,\tau'\in \Theta$, $\phi(\tau)=\phi(\tau')$ a.s.  on
$\{\tau=\tau'\}$.

Moreover, for $p \in [1, +\infty]$ fixed, we say that an admissible family $\phi$ is $p$-integrable, if for all $\tau \in \Theta$, $\phi(\tau)$ is in $L^{p}$.
\end{definition}
Let $\phi=(\phi(\tau), \, \tau \in \Theta)$ be an admissible family. For a stopping time $\tau$  of the form \eqref{form}, we have 
\begin{equation}\label{formula}
\phi(\tau)= \sum_{k=0}^{+\infty} \phi(\theta_k ){\bf 1}_{A_k}+\phi(T){\bf 1}_{\bar A} \quad {\rm a.s.}
\end{equation}
Given two admissible families $\phi=(\phi(\tau), \, \tau \in \Theta)$ and $\phi'=(\phi'(\tau), \, \tau \in \Theta)$, we 
say that $\phi$ is {\em equal to} $\phi'$ and write $\phi = \phi'$ if, for all $\tau \in \Theta$, 
$\phi(\tau) =\phi'(\tau)$ a.s. We 
say that $\phi$ {\em dominates} $\phi'$ and write $\phi \geq \phi'$
if, for all $\tau \in \Theta$, 
$\phi(\tau) \geq \phi'(\tau)$ a.s.\\
The following remark is  worth noting, as a consequence of the admissibility.  
\begin{Remark}\label{Rmk_on_admissibility}
Let $\phi=(\phi(\tau), \, \tau \in \Theta)$ be an admissible family. Let $\tau\in\Theta$ and let $(\tau_n)\in\Theta^{\N}$  be such that for (almost) each $\omega\in\Omega,$ there exists $n_0=n_0(\omega)$ (depending on $\omega$) satisfying, for all $n\geq n_0(\omega)$, $\tau_n(\omega)=\tau(\omega). $ Then, for all $n\geq n_0(\omega)$, $\phi(\tau_n)(\omega)=\phi(\tau)(\omega)$.\\ We show this by the following reasoning:  for each fixed $n\in\N$, let $\mathcal C_n:=\{\tau_n=\tau\}.$ For each fixed $m\in\N$, let
$\mathcal A_m:=\cap_{n\geq m}\mathcal C_n=\cap_{n\geq m}\{\tau_n=\tau\}.$ Note that  the set  $\mathcal A_m$ might be empty. We have $\cup_{m\in\N}\mathcal A_m=\Omega.$ Moreover, by the admissibility of $\phi$, we have, for each fixed $n\in\N$, $\phi(\tau_n)=\phi(\tau)$, on $\mathcal C_n=\{\tau_n=\tau \}.$ Hence, for each fixed  $m\in\N$,
 \begin{equation}\label{Eq_small_property}
  \text{ for all }n\geq m, \phi(\tau_n)=\phi(\tau)  \text{  on } \mathcal A_m=\cap_{n\geq m}\mathcal C_n.
 \end{equation} 
  Let $\omega\in\Omega$. By assumption, there exists  $n_0=n_0(\omega)$ such  that $\omega\in\mathcal A_{n_0}.$ By property \eqref{Eq_small_property} (applied with $m=n_0$), for all $n\geq n_0$, $\phi(\tau_n)(\omega)=\phi(\tau)(\omega),$ which is the desired conclusion.  

\end{Remark}

%
%

\subsection{The optimisation problem} \label{optimisation_problem}
Let $p \in [1, +\infty]$ be fixed.\\
Let  $\xi = (\xi(\tau), \, \tau \in \Theta)$ be \textbf{$p$-integrable admissible  family} modelling an agent's  dynamic financial position.
\begin{Remark} 
For example, the family  $\xi$ can be defined 
via a given  progressive process $(\xi_t)_{t\in[0,T]}$, corresponding to a given dynamic financial position process. For each $\tau$ $\in$ $\Theta$, we set  $\xi(\tau) := \xi_\tau$.  The  family of random variables $\xi=(\xi(\tau), \, \tau \in \Theta)$ can be shown to be admissible. If for each $k \in \mathbbm{N}$, $\xi_{\theta_{k}} \in L^{p}$, and $\xi_T\in L^p$, then the admissible family $\xi$ is $p$-integrable. The financial interpretation of this example is as follows: 
the agent  can choose his/her strategy only among the stopping times in $\Theta$,  that is, among the stopping times which, for almost each $\omega$, have values in the finite  grid  $\{0, \theta_1(\omega), \ldots, \theta_{n(\omega)}(\omega)=T\}$, where $n(\omega)$ depends on $\omega$, or in the infinite countable grid $\{0, \theta_1(\omega), \ldots, \theta_{n}(\omega), \theta_{n+1}(\omega),\ldots, T\}.$  In this example, the financial position which is actually taken into account in the problem
corresponds to the values of the process $(\xi_t)$ only at times $0, \theta_1,..., \theta_{n}, \theta_{n+1},..., T$.
\end{Remark}
\noindent
The \textit{minimal risk at time} $0$ over all Bermudan stopping strategies  is defined by:
\begin{equation}\label{optimal_stopping_problem_0bis}
\tilde V(0):= \inf_{\tau\in \Theta} \tilde \rho_{0,\tau}(\xi(\tau))=  -V(0),
\end{equation}
where
\begin{equation}\label{optimal_stopping_problembis}
V(0):=\sup_{\tau\in\Theta} \rho_{0,\tau}[\xi(\tau)],
\end{equation}
and where $\rho_{0,\tau}[\cdot]=-\tilde\rho_{0,\tau}[\cdot].$\\
[0.2cm]
Let $p \in [1, +\infty]$. We introduce the following properties  on the non-linear operators $\rho_{S, \tau}[\cdot]$, which will appear in the sequel. \\
[0.2cm] 
For
$S\in\Theta$, $S'\in\Theta$, $\tau\in\Theta$, for $\eta$, $\eta_{1}$ and $\eta_2$ in $L^p(\cf_\tau)$, for $\xi=(\xi(\tau))$ an admissible p-integrable family: 
\begin{compactenum}[(i)]
\item[(i)] $\rho_{S,\tau}: L^p(\cf_\tau)\longrightarrow L^p(\cf_S)$ 
\item[(ii)] \emph{(admissibility)} $\rho_{S,\tau}[\eta]=\rho_{S',\tau}[\eta]$ a.s.  on $\{S=S'\}$. 
\item[(iii)] \emph{(knowledge preservation)}
$\rho_{\tau,S}[\eta]=\eta, $
for all $\eta\in L^p(\cf_S)$, all $\tau\in\Theta_S.$
\item[(iv)] \emph{(monotonicity)}  $\;\rho_{S,\tau}[\eta_1]\leq \rho_{S,\tau}[\eta_2]$ a.s., if  $\eta_1\leq \eta_2$ a.s. 
\item[(v)] \emph{(consistency)}  $\;\rho_{S,\theta}[\rho_{\theta,\tau}[\eta]]= \rho_{S,\tau}[\eta]$, for all $S, \theta, \tau$ in $\Theta$ such that $S\leq \theta\leq \tau$ a.s. 
\item [(vi)] \emph{("generalized zero-one law") }$\; I_A\rho_{S,\tau}[\xi(\tau)] =I_A\rho_{S,\tau'}[\xi(\tau')],$ for all $A\in\cf_S$, $\tau\in \Theta_{S}$, $\tau'\in\Theta_{S}$ such that $\tau=\tau'$ on $A$.  

\item[(vii)] \emph{(monotone Fatou property with respect to terminal condition)}\\
 $\rho_{S, \tau}[\eta] \leq \liminf_{n \to +\infty} \rho_{S, \tau}[\eta_{n}]$, for $(\eta_{n}), \eta$ such that $(\eta_{n})$ is non-decreasing, $\eta_{n} \in L^{p}(\cf_{\tau})$, $\sup_{n}\eta_{n} \in L^{p}$, and $\lim_{n \to +\infty} \uparrow \eta_{n} = \eta$ a.s.\\
\end{compactenum}
\noindent
Fatou property is often assumed in the literature on risk measures (particularly in the case where $p = +\infty$).\\
Note also that if $\rho$ satisfies monotonocity (iv) and monotone Fatou property with respect to  terminal condition (vii), then $\rho_{S, \tau}[\eta]= \lim_{n \to +\infty} \rho_{S, \tau}[\eta_{n}]$, for $(\eta_{n}), \eta$ such that $(\eta_{n})$ is non-decreasing, $\eta_{n} \in L^{p}(\cf_{\tau})$, $\sup_{n}\eta_{n} \in L^{p}$, and $\lim_{n \to +\infty} \uparrow \eta_{n} = \eta$ a.s. Indeed, by monotonicity of $\rho_{S, \tau}[\cdot]$, we have $\rho_{S, \tau}[\eta_n]\leq \rho_{S, \tau}[\eta]$; hence, $\limsup_{n\to+\infty}\rho_{S, \tau}[\eta_n]\leq \rho_{S, \tau}[\eta]. $ On the other hand, by   
(vii), $\rho_{S, \tau}[\eta]\leq \liminf_{n\to+\infty}\rho_{S, \tau}[\eta_n]. $ Hence, $\rho_{S, \tau}[\eta]= \lim_{n\to+\infty}\rho_{S, \tau}[\eta_n]$. Such type of property  is also known in the literature (e.g. risk measures) as continuity from below. \\
Let us emphasize that  no assumptions of convexity (or concavity) or translation invariance of the non-linear operators $\rho$ are made.

%
%
\subsection{$(\Theta,\rho)$-Snell envelope family and optimal stopping}\label{subsect2}

As is usual in optimal control, we embed the above optimization problem \eqref{optimal_stopping_problembis} in a larger class of problems by 
considering for each $\nu \in \Theta$, the random variable $V(\nu)$, where 
\begin{equation}\label{optimal_stopping_problem_kbis}
V(\nu):=\esssup_{\tau\in \Theta_{\nu}} \rho_{\nu,\tau}[\xi(\tau)].
\end{equation} 
 We note that, if $\rho$ satisfies the property of knowledge preservation (property (iii)), then $V(T)=\rho_{T,T}[\xi(T)]=\xi(T).$
\begin{lemma}\label{Lemma_admissible}(Admissibility of $V$)
Under the assumption of admissibility (ii) and ``generalized zero-one law'' (vi) on the non-linear operators, the family of random variables
$V:=  (V(\nu), \, \nu \in \Theta)$ defined in \eqref{optimal_stopping_problem_kbis} is \emph{admissible} in the sense of Definition \ref{def.admi}.

\end{lemma}
\noindent
The proof uses arguments similar to those of  Lemma 8.1 in \cite{Grigorova-3}, combined with some properties of the non-linear operators $\rho$.

\begin{proof} 
Property 1. of the definition of admissibility follows from the definition of the essential supremum, the random variables of the family   $(\rho_{\nu,\tau}[\xi(\tau)], \tau\in\Theta_\nu)$ being $\cf_\nu$-measurable. \\
Let us prove Property 2. Let $\nu$ and $\nu'$ be two stopping times in $\Theta$. We set $A:=\{\nu=\nu'\}$ and we show that $V(\nu)=V(\nu')$, $P$-a.s. on $A$. 
We have   
\begin{equation}\label{eq_toto1}
\begin{aligned}
\I_AV(\nu)&=\I_A\esssup_{\tau\in \Theta_{\nu}} \rho_{\nu,\tau}[\xi(\tau)] =\esssup_{\tau\in \Theta_{\nu}} \I_A\rho_{\nu,\tau}[\xi(\tau)]=\esssup_{\tau\in \Theta_{\nu}} \I_A\rho_{\nu',\tau}[\xi(\tau)],
\end{aligned}
\end{equation}
where we have used the admissibility property on $\rho$ for the last equality.\\
Let $\tau\in\Theta_\nu$. We set $\tau_A:= \tau\I_A+T\I_{A^c}$. We note that $\tau_A\in\Theta_{\nu'}$ and $\tau_A=\tau$ p.s. on $A$. Using this, the admissibility of the family $\xi$, and the generalized zero-one law property of $\rho$, we get
$
\I_A\rho_{\nu',\tau}[\xi(\tau)]=\I_A\rho_{\nu',\tau_A}[\xi(\tau_A)]\leq \I_AV(\nu').
$
As $\tau\in\Theta_\nu$ is arbitrary, we conclude that 
$\esssup_{\tau\in \Theta_{\nu}}\I_A\rho_{\nu',\tau}[\xi(\tau)]\leq \I_AV(\nu').$
Combining this inequality with \eqref{eq_toto1} gives 
$\I_AV(\nu)\leq \I_AV(\nu').$ We obtain the converse inequality by interchanging the roles of $\nu$ and $\nu'$. 
\end{proof}
\noindent
Under the assumptions of the above lemma, the following remark holds true.

\begin{Remark}\label{Rmk_consequence_admissible}
As a consequence of the \textbf{admissibility} of the value family $V$, we have: 
for each  $k\in\N, $ it holds
$V(\nu)= V (\theta_k)$ a.s. on $\{\nu=\theta_k\}$ and $V(\nu)=V(T)$ a.s. on $\{\nu=T\}$.  Hence, under the assumptions of Lemma \ref{Lemma_admissible}, for $\nu\in\Theta$ of the form $\nu=  \sum_{k=0}^{+\infty} \theta_k {\bf 1}_{A_k}+ T{\bf 1}_{\bar A} $, we have $V(\nu)=\sum_{k=0}^{+\infty} V(\theta_k ){\bf 1}_{A_k}+V(T){\bf 1}_{\bar A} $.
\end{Remark}

\begin{Remark}\label{TEXRemark33}
1. Under the assumption of knowledge preservation (iii) on $\rho$, we have $V(\theta_k) \geq \xi(\theta_k)$, for each $k\in\N.$\\
Indeed, $V(\theta_{k}) = \esssup_{\tau \in \Theta_{\theta_{k}}}\rho_{\theta_{k}, \tau}[\xi(\tau)] \geq \rho_{\theta_{k}, \theta_{k}}[\xi(\theta_{k})]$, and by the property (iii) of the non-linear operators, we have $\rho_{\theta_{k}, \theta_{k}}[\xi(\theta_{k})] = \xi(\theta_{k})$. Hence, $V(\theta_{k}) \geq \xi(\theta_{k})$.\\
[0.2cm]
2.  If, moreover, $\rho$ satisfies the properties of  admissibility (ii) and "generalized" zero-one law (vi), then, for each $\tau\in\Theta,$  $V(\tau) \geq \xi(\tau)$. \\
This follows from the first statement of the remark, and from the admissibility of $\xi$ and that of $V$ (cf. Lemma \ref{Lemma_admissible} and Remark \ref{Rmk_consequence_admissible}).  

\end{Remark}

\noindent

\noindent
\noindent
Now, let us introduce the notion of \emph{$(\Theta, \rho)$-(super)martingale family}.
\begin{definition}\label{def_supermartingaleBIS}
Let $\phi=(\phi(\tau), \, \tau \in \Theta)$ be a \emph{$p$-integrable admissible} family.
 We say that $\phi$
 is a \emph{$(\Theta, \rho)$-supermartingale} (resp. \emph{$(\Theta, \rho)$--martingale}) \emph{ family} if 
 for all $\sigma, \tau$ in $\Theta$ such that $\sigma\leq\tau$ a.s., we have 
$$\rho_{\sigma,\tau}[\phi({\tau})]\leq \phi({\sigma}) \text{ (resp. }=\phi(\sigma)) \text{ a.s.} $$
%
\end{definition}
\noindent
We introduce the following integrability assumption on $V$, which is assumed in the sequel. 
\begin{Assumption}\label{Hyp3}
For each   $\nu \in \Theta$, the random variable $V(\nu)$ is in $L^p$.
\end{Assumption}
\begin{Remark} \label{TEXRemaRk24}
Let $\rho$ satisfy the assumptions of admissibility (ii), knowledge preservation (iii), "generalized" zero-one law (vi), and monotonicity (iv). If the pay off family $\xi = (\xi(\tau))_{\tau \in \Theta}$ is $p$- integrable and dominated from above by a $p$-integrable $(\Theta, \rho)$-martingale $M$, then the value family $V$ satisfies the integrability Assumption \ref{Hyp3}.\\
Indeed, let $S \in \Theta$ be given. By Remark \ref{TEXRemark33}, Statement 2, $V(S) \geq \xi(S)$.\\ 
On the other hand, by assumption on $\xi$, for each $\tau \in \Theta_{S}$, $\xi(\tau) \leq M(\tau)$. Hence, by monotonicity of $\rho$, we have $\rho_{S, \tau}[\xi(\tau)] \leq \rho_{S, \tau}[M(\tau)] = M(S),$
where we have used the $(\Theta, \rho)$-martingale property of $M$ for the last equality.\\
So, $V(S) = \esssup_{\tau \in \Theta_{S}}\rho_{S, \tau}[\xi(\tau)] \leq M(S)$. Hence, we get $\xi(S) \leq V(S) \leq M(S)$, which proves that $V(S) \in L^{p}$. Therefore, Assumption \ref{Hyp3} is satisfied.

\end{Remark}
\noindent
We will see in Section \ref{SectioN3}, concrete examples for which this integrability assumption on $V$ is satisfied.

\begin{theorem}\label{Thm_Snell}
1. \emph{($(\Theta,\rho)$-supermartingale)} Under the assumptions of admissibility (ii), consistency (v), ``generalized zero-one law'' (vi) and monotone Fatou property with respect to the terminal condition (vii) on the non-linear operators, the value family $V$ is a $(\Theta, \rho)$-supermartingale family.\\
[0.3cm]
2. \emph{($(\Theta,\rho)$-Snell envelope)} If moreover the non-linear operators also satisfy the properties of knowledge preservation (iii) and monotonicity (iv), the value family $V$ is equal to the $(\Theta,\rho)$-Snell envelope of the family  $\xi$, that is,  the smallest $(\Theta,\rho)$-supermartingale family dominating the family $\xi= (\xi(\tau), \, \tau \in \Theta)$.

\end{theorem}

To prove this theorem, we first state a useful lemma.

\begin{lemma}\label{Maximizing sequence}(Maxmizing sequence lemma)
Under the assumption of ``generalized zero-one law'' (vi) on the non-linear operators, there exists a maximizing sequence for the value $V(S)$ of problem \eqref{optimal_stopping_problem_kbis}. 

\end{lemma}
\noindent
The proof of this lemma is similar to that of Lemma 2.3 in \cite{Grigorova-4} and is given for the convenience of the reader. 
\begin{proof}
It is sufficient to show that the family $(\rho_{S,\tau} [\xi(\tau)])_{\tau\in\Theta_S}$  is stable under pairwise maximization. The result then follows by a well-known property of the essential supremum. 
Let $\tau\in\Theta_S$ and $\tau'\in\Theta_S$. Set $A:=\{\rho_{S,\tau'} [\xi(\tau')]\leq \rho_{S,{\tau}} [\xi(\tau)]\} $ and  $\nu:=\tau \I_A+\tau'\I_{A^c}. $ Trivially, $A\in\mathcal{F}_S.$ Moreover,  $\nu\in\Theta_S$ (cf. properties of the set $\Theta_S$).  Also,   $\nu=\tau$ on $A$,  $\nu=\tau'$ on $A^c$.  By the "generalized zero-one law" of the non-linear operators $\rho$, we get
\begin{equation}\label{eq_computation_0}
\begin{aligned}
 \rho_{S,{\nu}} [\xi(\nu)]=\rho_{S,{\nu}} [\xi(\nu)]\I_A+\rho_{S,{\nu}} [\xi(\nu)]\I_{A^c}&=\rho_{S,{\tau}} [\xi(\tau)]\I_A+\rho_{S,{\tau'}} [\xi(\tau')]\I_{A^c}\\
 &=\max\big(\rho_{S,{\tau}} [\xi(\tau)],\rho_{S,{\tau'}} [\xi(\tau')]\big).
\end{aligned}
\end{equation}
This shows the stability under pairwise maximization of the value family (indexed by $\Theta_S$).  
\end{proof}
\noindent
Let us now show the theorem. The idea of the proof is similar to that of Theorem 8.2 in \cite{Grigorova-3}. The properties on $\rho$ being weakened here, we give the proof for clarity and completeness.

\begin{proof}[Proof of Theorem \ref{Thm_Snell}]
By Lemma \ref{Lemma_admissible}, the value family $V$  is admissible. By Assumption \ref{Hyp3}, the value family $V$ is $p$-integrable.\\
Let now $S\in\Theta$ and $\tau\in\Theta_S$. To show the $(\Theta,\rho)$-supermartingale property of the value family, it remains to show $\rho_{S,\tau}[V(\tau)]\leq V(S)$ a.s. By the maximizing sequence lemma (Lemma \ref{Maximizing sequence}), there exists a sequence $(\tau_{p}) \in (\Theta_{\tau})^{\mathbbm{N}}$, such that $V(\tau) = \lim_{p \to +\infty}\uparrow\rho_{\tau, \tau_{p}}[\xi(\tau_{p})]$. Hence, we have
$$\rho_{S, \tau}[V(\tau)] = \rho_{S, \tau}[\lim_{p \to +\infty}\uparrow\rho_{\tau, \tau_{p}}[\xi(\tau_{p})]] \leq \liminf_{p \to +\infty}\rho_{S, \tau}[\rho_{\tau, \tau_{p}}[\xi(\tau_{p})]],$$
where we have used the monotone Fatou property with respect to terminal condition (vii) to obtain the inequality. By the consistency property, we have
$$\liminf_{p \to +\infty}\rho_{S, \tau}[\rho_{\tau, \tau_{p}}[\xi(\tau_{p})]] = \liminf_{p \to +\infty}\rho_{S, \tau_{p}}[\xi(\tau_{p})] \leq V(S),$$
the last inequality being due to $\Theta_{\tau} \subset \Theta_{S}$. We conclude that $\rho_{S, \tau}[V(\tau)] \leq V(S)$. Hence, the value family $V$ is a $(\Theta, \rho)$-supermartingale family. This proves Statement 1 of the theorem.\\
[0.3cm]
Let us now show Statement 2. By Remark \ref{TEXRemark33}, Statement 2,  we have $V \geq \xi$. By Statement 1, we have that $V$ is a $(\Theta, \rho)$-supermartingale. It remains to show that $V$ is the smallest. Let $(V'(\tau))$ be another $(\Theta, \rho)$-supermartingale family, such that, for each $\tau \in \Theta$,  $V'(\tau) \geq \xi(\tau)$ (a.s.). Let $S \in \Theta$, $\tau \in \Theta_{S}$. By the monotonicity of the non-linear operators $\rho$, we have
$$\rho_{S, \tau}[V'(\tau)] \geq \rho_{S, \tau}[\xi(\tau)].$$
On the other hand, as $(V'(\tau))$ is a $(\Theta, \rho)$-supermartingale family, 
 we have 
 $V'(S) \geq \rho_{S, \tau}[V'(\tau)].$
Hence, $V'(S) \geq \rho_{S, \tau}[V'(\tau)] \geq \rho_{S, \tau}[\xi(\tau)].$ By taking the essential supremum over $\tau \in \Theta_{S}$ in this  inequality, we get
$$V'(S)  \geq \esssup_{\tau \in \Theta_{S}} \rho_{S, \tau}[\xi(\tau)] = V(S) \text{ (a.s.)}.$$
The proof is complete. 
\end{proof}

\subsection{The strict value family and the Dynamic Programming Principle (DPP)} \label{SUBset4}
\begin{definition}[Dynamic Programming Principle] We say that an admissible $p$-integrable family satisfies the Dynamic Programming Principle (abridged DPP), if the following property holds true:\\ 
 For all $k\in\mathbb{N}$,
\begin{equation}\label{Eq_def_DPP}
\phi(\theta_k)=\max\big(\xi(\theta_k), \rho_{\theta_k,\theta_{k+1}}[\phi(\theta_{k+1})]\big),
\end{equation}
and $\phi(T)=\xi(T).$
\end{definition}
\noindent
The purpose of this sub-section is to investigate under which assumptions on $\rho$, the DPP holds. To do this, we are first interested in ``what happens on the right of $V(\theta_{k})$'', for each $k \in \mathbbm{N}$.\\
[0.2cm]
Let $k \in \mathbbm{N}$ be fixed. We define
$$\Theta_{\theta_{k}^{+}} \coloneqq \{\tau \in \Theta_{\theta_{k}}: \tau > \theta_{k} \text{ on } \{\theta_{k} < T\} \text{ and } \tau = T \text{ on } \{\theta_{k} = T\}\},$$
and we define \textbf{the strict value} $V^{+}(\theta_{k})$ at $\theta_{k}$ by:
$$V^{+}(\theta_{k}) \coloneqq \esssup_{\tau \in \Theta_{\theta_{k}}^{+}} \rho_{\theta_{k}, \tau}[\xi(\tau)].$$
\begin{Remark}
We have $\Theta_{\theta_{k}^{+}} = \Theta_{\theta_{k+1}}$.\\
Indeed, let $\tau \in \Theta_{\theta_{k}^{+}}$. Then $\tau$ can be written as: 
$$\tau = \sum^{+\infty}_{i = k+1} \theta_{i}\mathbbm{1}_{A_{i}\cap\{\theta_{k}<T\}} +T \times \mathbbm{1}_{\bar A\cap \{\theta_{k}<T\}}+ T \times \mathbbm{1}_{\{\theta_{k}=T\}},$$
where $\{(A_{i})_{i \geq k+1},\bar A\}$ is a partition of $\Omega$ such that for each  $i \geq k+1$, $A_{i} \in \mathcal{F}_{\theta_{i}}$, and $\bar A\in\cf_T.$\\
We set $B_{i} \coloneqq A_{i} \cap \{\theta_{k} < T\}$, for $i \geq k+1$, $\bar B\coloneqq \bar A\cap \{\theta_{k} < T\} $and $B_{k} \coloneqq \{\theta_{k} = T\}$. We have $\{(B_{i})_{i \geq k},\bar B\}$ form a partition of $\Omega$; for each $i \geq k$, $B_{i}$ is $\mathcal{F}_{\theta_{i}}$-measurable, and $\bar B$ is $\cf_T$-measurable.  Moreover, $\tau \geq \theta_{k+1}$ (indeed, $\tau \geq \theta_{k+1}$ on $\{\theta_{k}<T\}$ and $\tau = T = \theta_{k} = \theta_{k+1}$ on $\{\theta_{k} = T\}$). Hence, $\tau \in \Theta_{\theta_{k+1}}$.\\
Conversely, let $\tau \in \Theta_{\theta_{k+1}}$; then, $\tau$ can be written as:
\begin{align*}
\tau &= \sum_{i=k+1}^{+\infty}\theta_{i}\mathbbm{1}_{A_{i} \cap \{\theta_{k} < T\}} + T\mathbbm{1}_{\bar A \cap \{\theta_{k} < T\}}+\sum_{i=k+1}^{+\infty}\theta_{i}\mathbbm{1}_{A_{i} \cap \{\theta_{k} = T\}}+T\mathbbm{1}_{\bar A \cap \{\theta_{k} = T\}} \\
&= \sum_{i=k+1}^{+\infty}\theta_{i}\mathbbm{1}_{A_{i} \cap \{\theta_{k} < T\}} + T\mathbbm{1}_{\bar A \cap \{\theta_{k} < T\}}+ \sum_{i=k+1}^{+\infty}T\mathbbm{1}_{A_{i} \cap \{\theta_{k} = T\}}+ T\mathbbm{1}_{\bar A \cap \{\theta_{k} = T\}}.
\end{align*}
Hence, $\tau \in \Theta_{\theta_{k^{+}}}$.

\end{Remark}
\noindent
Due to this remark, we get
\begin{equation}\label{Eq_right}
V^{+}(\theta_{k}) = \esssup_{\tau \in \Theta_{\theta_{k}}^{+}} \rho_{\theta_{k}, \tau}[\xi(\tau)] = \esssup_{\tau \in \Theta_{\theta_{k+1}}} \rho_{\theta_{k}, \tau}[\xi(\tau)].
\end{equation}

\begin{lemma}\label{TEXlemma1}
Under the assumption of ``generalized zero-one law'' (vi) on the non-linear operators, there exists a maximizing sequence for $V^{+}(\theta_{k})$.

\end{lemma}

\begin{proof}
The proof of this lemma is similar to the proof of the existence of a maximizing sequence for $V(\theta_{k})$, and is left to the readers. (We also refer to the proof of Lemma 2.3 in \cite{Grigorova-4} for similar arguments).
\end{proof}
\noindent
The following proposition establishes that the strict value $V^{+}(\theta_{k})$ at $\theta_{k}$ is equal to the non-linear evaluation from $\theta_{k}$ perspective of the value $V(\theta_{k+1})$.

\begin{proposition} \label{TEXproposition1}
Under the assumptions of monotonicity (iv), consistency (v), ``generalized zero-one law'' (vi) and monotone Fatou property with respect to the terminal condition (vii) on the non-linear operators, we have
$$V^{+}(\theta_{k}) = \rho_{\theta_{k}, \theta_{k+1}}[V(\theta_{k+1})].$$

\end{proposition}

\begin{proof}
First we show that $V^{+}(\theta_{k}) \leq \rho_{\theta_{k}, \theta_{k+1}}[V(\theta_{k+1})]$.\\
[0.2cm]
By Lemma \ref{TEXlemma1}, there exists a maximizing sequence $(\tau_{m}) \in (\Theta_{\theta_{k+1}})^{\mathbbm{N}}$ such that 
$$V^{+}(\theta_{k}) = \lim_{m \to +\infty} \uparrow \rho_{\theta_{k}, \tau_{m}}[\xi(\tau_{m})].$$
Now, by using the consistency property of the non-linear evaluations, we get
\begin{equation}\label{TEXequation111}
V^{+}(\theta_{k}) = \lim_{m \to +\infty} \uparrow \rho_{\theta_{k}, \tau_{m}}[\xi(\tau_{m})] = \lim_{m \to +\infty} \uparrow \rho_{\theta_{k}, \theta_{k+1}}[\rho_{\theta_{k+1}, \tau_{m}}[\xi(\tau_{m})]].
\end{equation}
For each $m \in \mathbbm{N}$, we have  
$$\rho_{\theta_{k+1}, \tau_{m}}[\xi(\tau_{m})] \leq \esssup_{\tau \in \Theta_{\theta_{k+1}}}\rho_{\theta_{k+1}, \tau}[\xi(\tau)] = V(\theta_{k+1}).$$
Then, by the monotonicity property of $\rho_{\theta_{k}, \theta_{k+1}}[\cdot]$, we get
$$\rho_{\theta_{k}, \theta_{k+1}}[\rho_{\theta_{k+1}, \tau_{m}}[\xi(\tau_{m})]] \leq \rho_{\theta_{k}, \theta_{k+1}}[V(\theta_{k+1})].$$
Hence, we have 
$$\lim_{m \to +\infty} \uparrow \rho_{\theta_{k}, \theta_{k+1}}[\rho_{\theta_{k+1}, \tau_{m}}[\xi(\tau_{m})]] \leq \lim_{m \to +\infty} \uparrow \rho_{\theta_{k}, \theta_{k+1}}[V(\theta_{k+1})] = \rho_{\theta_{k}, \theta_{k+1}}[V(\theta_{k+1})].$$
We conclude, combining this with  \eqref{TEXequation111}, that 
$$V^{+}(\theta_{k}) \leq \rho_{\theta_{k}, \theta_{k+1}}[V(\theta_{k+1})].$$
Now, let us show the converse inequality. By Lemma \ref{Maximizing sequence}, there also exists a maximizing sequence $(\tau_{m}') \in (\Theta_{\theta_{k+1}})^{\mathbbm{N}}$ such that 
$$V(\theta_{k+1}) = \lim_{m \to +\infty} \uparrow \rho_{\theta_{k+1}, \tau_{m}'}[\xi(\tau_{m}')].$$
Hence, 
$$\rho_{\theta_{k}, \theta_{k+1}}[V(\theta_{k+1})] = \rho_{\theta_{k}, \theta_{k+1}}[\lim_{m \to +\infty} \uparrow \rho_{\theta_{k+1}, \tau_{m}'}[\xi(\tau_{m}')]].$$
We first use the monotone Fatou property with respect to the terminal condition of the non-linear operator $\rho_{\theta_{k}, \theta_{k+1}}[\cdot]$; then, we apply the consistency property of the non-linear operators to get:
\begin{align*}
\rho_{\theta_{k}, \theta_{k+1}}[V(\theta_{k+1})] &\leq \liminf_{m \to +\infty}\rho_{\theta_{k}, \theta_{k+1}}[\rho_{\theta_{k+1}, \tau_{m}'}[\xi(\tau_{m}')]]\\
&= \liminf_{m \to +\infty}\rho_{\theta_{k}, \tau_{m}'}[\xi(\tau_{m}')]\\
&\leq \esssup_{\tau \in \Theta_{\theta_{k+1}}} \rho_{\theta_{k}, \tau}[\xi(\tau)] = V^{+}(\theta_{k}),
\end{align*}
where we have used Eq. \eqref{Eq_right} to obtain the last equality. 
Hence, 
$$\rho_{\theta_{k}, \theta_{k+1}}[V(\theta_{k+1})] = V^{+}(\theta_{k}).$$
The proof is complete.
\end{proof}

\begin{proposition} \label{TEXproposition2}
Under the assumptions (iii) and ``generalized zero-one law'' (vi) on the non-linear operators, we have
$$V(\theta_{k}) = \xi(\theta_{k}) \vee V^{+}(\theta_{k}).$$

\end{proposition}

\begin{proof}
By Remark \ref{TEXRemark33}, first statement,  which can be applied as $\rho$ satisfies property (iii), we have $V(\theta_{k}) \geq \xi(\theta_{k})$. On the other hand, since $\Theta_{\theta_{k+1}} \subset \Theta_{\theta_{k}}$, we have $V(\theta_{k}) \geq V^{+}(\theta_{k})$. By combining these two inequalities, we get $V(\theta_{k}) \geq \xi(\theta_{k}) \vee V^{+}(\theta_{k})$. It remains to show the converse inequality. Let $\tau \in \Theta_{\theta_{k}}$. We define $\bar{\tau} = \tau \mathbbm{1}_{\{\tau > \theta_{k}\}} + T \mathbbm{1}_{\{\tau \leq \theta_{k}\}}$. As $\bar{\tau} \in \Theta_{\theta_{k}}^{+}$, we have 
$$V^{+}(\theta_{k}) = \esssup_{\tau \in \Theta_{\theta_{k}}^{+}} \rho_{\theta_{k}, \tau}[\xi(\tau)] \geq \rho_{\theta_{k}, \bar{\tau}}[\xi(\bar{\tau})].$$
Hence, we have
\begin{equation} \label{TEXequation222}
\mathbbm{1}_{\{\tau > \theta_{k}\}} \rho_{\theta_{k}, \bar{\tau}}[\xi(\bar{\tau})] \leq \mathbbm{1}_{\{\tau > \theta_{k}\}} V^{+}(\theta_{k}).
\end{equation}
Moreover, on the set $\{\tau > \theta_{k}\}$, we have $\bar{\tau} = \tau$, so 
 the ``generalized zero-one law'' gives
\begin{equation} \label{TEXequation333}
\mathbbm{1}_{\{\tau > \theta_{k}\}} \rho_{\theta_{k}, \bar{\tau}}[\xi(\bar{\tau})] = \mathbbm{1}_{\{\tau > \theta_{k}\}} \rho_{\theta_{k}, \tau}[\xi(\tau)].
\end{equation}
By combining \eqref{TEXequation222} and \eqref{TEXequation333}, we get
\begin{equation} \label{TEXequation444}
\mathbbm{1}_{\{\tau > \theta_{k}\}} \rho_{\theta_{k}, \tau}[\xi(\tau)] \leq \mathbbm{1}_{\{\tau > \theta_{k}\}} V^{+}(\theta_{k}).
\end{equation}
On the other hand, as $\tau \in \Theta_{\theta_{k}}$, we have
$$\rho_{\theta_{k}, \tau}[\xi(\tau)] = \mathbbm{1}_{\{\tau = \theta_{k}\}} \rho_{\theta_{k}, \tau}[\xi(\tau)] + \mathbbm{1}_{\{\tau > \theta_{k}\}} \rho_{\theta_{k}, \tau}[\xi(\tau)].$$
By using the ``generalized zero-one law'' and property (iii) of the non-linear operator $\rho_{\theta_{k}, \tau}[\cdot]$, we get
\begin{equation} \label{TEXequation555}
\mathbbm{1}_{\{\tau = \theta_{k}\}} \rho_{\theta_{k}, \tau}[\xi(\tau)] = \mathbbm{1}_{\{\tau = \theta_{k}\}} \rho_{\theta_{k}, \theta_{k}}[\xi(\theta_{k})] = \mathbbm{1}_{\{\tau = \theta_{k}\}} \xi(\theta_{k}).
\end{equation}
From Eqs. \eqref{TEXequation444} and \eqref{TEXequation555}, we get
\begin{align*}
\rho_{\theta_{k}, \tau}[\xi(\tau)] &= \mathbbm{1}_{\{\tau = \theta_{k}\}} \xi(\theta_{k}) + \mathbbm{1}_{\{\tau > \theta_{k}\}} \rho_{\theta_{k}, \tau}[\xi(\tau)]\\
&\leq \mathbbm{1}_{\{\tau = \theta_{k}\}} \xi(\theta_{k}) + \mathbbm{1}_{\{\tau > \theta_{k}\}} V^{+}(\theta_{k}) = \xi(\theta_{k}) \vee V^{+}(\theta_{k}). 
\end{align*}
Now, by taking the essential supremum over $\tau \in \Theta_{\theta_{k}}$, we get $V(\theta_{k}) \leq \xi(\theta_{k}) \vee V^{+}(\theta_{k})$. Hence, the proof is complete.
\end{proof}
\noindent
We refer to \cite{Quenez-Kobylanski} for a similar approach the one in the above proposition in the linear case.\\

By combining Proposition \ref{TEXproposition1} and Proposition \ref{TEXproposition2}, we get:

\begin{theorem}[DPP]\label{TEXtheorem1}
Under the assumptions of knowledge preservation (iii), monotonicity (iv), consistency (v), ``generalized zero-one law'' (vi) and monotone Fatou property with respect to the terminal condition (vii) on the non-linear operators, the value family $V$ satisfies the DPP: 
$$\text{for each } k \in \mathbbm{N}, \; V(\theta_{k}) = \xi(\theta_{k}) \vee \rho_{\theta_{k}, \theta_{k+1}}[V(\theta_{k+1})], \text{ and } V(T)=\xi(T).$$

\end{theorem}


\subsection{Optimal stopping times} \label{SUBset5}

For each $k$, let us define the random variable $\nu_k$ 
\begin{equation}\label{nu_zerobis} \nu_k:=
{\rm ess} \inf \;{\cal A}_{k} \quad \mbox{where} \quad 
{\cal A}_{k}:=\{\,\tau \in \Theta_{\theta_k}\,: \,V({\tau}) = \xi({\tau}) \,\,\mbox {a.s.} \,\}.
\end{equation}
As $T<\infty$, under property (iii) on $\rho$,  the set ${\cal A}_{k}$ is clearly non-empty (as $V(T)=\xi(T)$ in this case). Moreover,  it is clearly stable by pairwise minimization. 
Hence, by classical properties of the essential infimum, there exists a non increasing  sequence $(\tau_n)$ in ${\cal A}_{k}$ such that  
$\lim_{n \rightarrow + \infty} \tau_n= \nu_k$ a.s. In particular, $\nu_k$ is a 
stopping time and $T \geq \nu_k \geq \theta_k$ a.s., and $\nu_k\in\Theta_{\theta_k}$
(by stability of   $\Theta_{\theta_k}$ when passing to a monotone limit).\\
[0.3cm] 
In the following theorem, we show that, under suitable assumptions,   the stopping time $\nu_k$ defined in \eqref{nu_zerobis}  is optimal for the optimization problem \eqref{optimal_stopping_problem_kbis} at time $\nu= \theta_k$.\\
We introduce the following assumption on the value family $V$. 
\begin{Assumption}\label{HypA}
We assume that the value family  $V$ is left-upper-semicontinuous (LUSC) along the sequence $(\theta_n\wedge \nu_k)_{n\in\N}$, that is,
\begin{equation} \label{TEXEQuatioNN266}
\limsup_{n \to +\infty} V(\theta_{n} \wedge \nu_{k}) \leq V(\nu_{k}).
\end{equation}
\end{Assumption}
\begin{Remark} 
Assumption \ref{HypA} is trivially satisfied in the following particular case on $\Theta$:
Besides  the assumptions $(a)$ and $(b)$ on $\Theta$, the additional assumption   (c) is imposed, namely:\\
(c)For almost all  $\omega$, there exits $n_0= n_0(\omega)$ (depending on $\omega$) such that $\theta_n(\omega)=T$, for all $n\geq n_0$. In other words, for almost all $\omega$, there exists at most a finite number of time points $\theta_n(\omega)$ such that $\theta_n(\omega)<T$.\\
In this case, for all $n$ 
after a certain rank $\bar{n} = \bar{n}(\omega)$, we have $(\theta_{n}\wedge\nu_{k})(\omega) = \nu_{k}(\omega)$. Hence, as $V$ is admissible, we have, by Remark \ref{Rmk_on_admissibility}, for all $n \geq \bar{n}(\omega)$, $V(\theta_{n} \wedge \nu_{k})(\omega) = V(\nu_{k})(\omega)$. Hence, Assumption \ref{HypA} holds true.
\end{Remark}
We will see later on a further discussion on  Assumption \ref{HypA} in the case of the general $\Theta$, and conditions (on $\rho$ and on the pay-off family $\xi$) under which this assumption  is satisfied.   
\begin{theorem}[Optimality of $\nu_k$]\label{prop_martingale_optimality_principle_general} 
Let $k \in \mathbbm{N}$ and let $\nu_k$ be the stopping time defined by \eqref{nu_zerobis}. Let Assumption \ref{HypA} on $V$ be satisfied.  Let $\rho$ satisfy  the properties  of admissibility (ii), knowledge preservation (iii), monotonicity (iv), consistency (v), ``generalized zero-one law'' (vi), and monotone Fatou property with respect to the terminal condition (vii).
We assume additionally that $\rho$ satisfies the following property: 
\begin{itemize}
\item (left-upper-semicontinuity (LUSC) along Bermudan stopping times with respect to the terminal condition and the terminal time at $\nu_{k}$), that is,  
\begin{equation}\label{cond_on_rho}
\limsup_{n \to +\infty}\rho_{S, \tau_{n}}[\phi(\tau_{n})] \leq \rho_{S, \nu_{k}}[\limsup_{n \to +\infty} \phi(\tau_{n})],
\end{equation}
$ \text{ for each non-decreasing sequence }(\tau_{n}) \in \Theta_{S}^{\N} \text{ such that }\lim_{n \to +\infty} \uparrow \tau_{n} = \nu_{k} \text{ a.s.},$ and for each $p$-integrable admissible family $\phi$ such that $\sup_{n\in\N} |\phi(\tau_n)|\in L^p.$
\end{itemize} 
Then:
\begin{equation}\label{eq_optimal_generalbis}
V(\theta_k)=\rho_{\theta_k,\nu_k}[\xi(\nu_k)]=\esssup_{\nu\in \Theta_{\theta_k}} \rho_{\theta_k,\nu}[\xi(\nu)] \quad a.s.\,
\end{equation}

\end{theorem}
\noindent
Note that in the case where $\rho=(\rho_S[\cdot])_{S\in\Theta}$ does not depend on the second time index, the above additional property reduces to the LUSC of $\rho_S[\cdot]$ (with respect to the terminal condition) along Bermudan stopping sequences. 

\subsubsection{The $(\Theta,\rho)$-martingale property on a stochastic interval}
Before proving the theorem, we give several useful technical lemmas.\\
[0.2cm]
The first two clarify the $(\Theta,\rho)$-martingale structure on a stochastic interval in a more "handy" way. 
The third one deals with an "\textit{if}-condition" (optimality criterion) and an "\textit{only if}-condition" for optimality.  

\begin{lemma}\label{Lemma2525}
Let $\rho$  satisfy the consistency property (v). Let $\phi = (\phi(\nu))$ be a given square-integrable admissible family. Let $S \in \Theta$ and $\tau \in \Theta$ be such that $S \leq \tau$ a.s.  We assume that for any $\sigma \in \Theta$ such that $S \leq \sigma \leq \tau$ a.s., it holds
\begin{equation} \label{third}
\rho_{\sigma, \tau}[\phi(\tau)] = \phi(\sigma) \quad a.s. \,
\end{equation}
Then, $\phi$ is a $(\Theta, \rho)$-martingale on the stochastic interval $[S, \tau]$, that is, for any $\nu_{1} \in \Theta$, $\nu_{2} \in \Theta$, such that $S \leq \nu_{1} \leq \nu_{2} \leq \tau$ a.s., 
\begin{equation}
\rho_{\nu_{1}, \nu_{2}}[\phi(\nu_{2})] = \phi(\nu_{1}) \quad a.s. \,
\end{equation}

\end{lemma}

\begin{proof}
Let $\nu_{1} \in \Theta$, and $\nu_{2} \in \Theta$ be such that $S \leq \nu_{1} \leq \nu_{2} \leq \tau$ a.s.\\
[0.2cm]
Hence, by applying the equation \eqref{third} with $\sigma = \nu_{1}$ and by the consistency of the non-linear operators $\rho$ , we have
$$\phi(\nu_{1}) = \rho_{\nu_{1}, \tau}[\phi(\tau)] = \rho_{\nu_{1}, \nu_{2}}[\rho_{\nu_{2}, \tau}[\phi(\tau)]] \; a.s.$$
Then, by applying again the equation \eqref{third} with $\sigma = \nu_{2}$, we have
$$\rho_{\nu_{2}, \tau}[\phi(\tau)] = \phi(\nu_{2}).$$
Hence,
$\phi(\nu_{1}) = \rho_{\nu_{1}, \nu_{2}}[\rho_{\nu_{2}, \tau}[\phi(\tau)]] =  \rho_{\nu_{1}, \nu_{2}}[\phi(\nu_{2})] \; a.s.$
\end{proof}

\begin{definition}(Strictly monotone operator) 
Let $S \in \Theta, \tau \in \Theta_{S}$. We say that $\rho_{S, \tau}$ is strictly monotone if the following two conditions hold:\\
[0.2cm]
1. $\rho_{S, \tau}$ is monotone.\\
[0.2cm]
2. If $\eta_{1} \leq \eta_{2}$ and $\rho_{S, \tau}(\eta_{1}) = \rho_{S, \tau}(\eta_{2})$, then $\eta_{1} = \eta_{2}$.

\end{definition}

\begin{lemma}\label{Lemma2626}
We assume that the non-linear operators satisfy the properties of monotonicity (iv) and consistency (v). Assume  moreover that the non-linear operators $\rho$ are strictly monotone. 
Let $\phi = (\phi(\nu))$ be a given p-integrable admissible family). Let $S \in \Theta$ and $\tau \in \Theta$ be such that $S \leq \tau$ a.s. We assume that the two conditions hold:\\
[0.2cm]
1. $\phi$ is a $(\Theta, \rho)$-supermartingale family on $[S, \tau]$;\\
[0.2cm]
2. $\phi(S) = \rho_{S, \tau}[\phi(\tau)]$ a.s.\\
[0.2cm]
Then, for any $\sigma \in \Theta$ such that $S \leq \sigma \leq \tau$ a.s., $\rho_{\sigma, \tau}[\phi(\tau)] = \phi(\sigma)$ a.s.

\end{lemma}

\begin{proof}
Let $\sigma \in \Theta$, such that $S \leq \sigma \leq \tau$ a.s.\\
By applying condition 2 of the lemma and the consistency of the non-linear operators $\rho$, we have
$$\phi(S) = \rho_{S, \tau}[\phi(\tau)] = \rho_{S, \sigma}[\rho_{\sigma, \tau}[\phi(\tau)]] \; a.s.$$
On the other hand, since $\phi$ is a $(\Theta, \rho)$-supermartingale family on $[S, \tau]$ (by condition 1), and since the non-linear operator $\rho_{S, \sigma}$ is monotone, we have
$$\rho_{S, \sigma}[\rho_{\sigma, \tau}[\phi(\tau)]] \leq \rho_{S, \sigma}[\phi(\sigma)] \leq \phi(S) \; a.s.$$
By combining the previous two equations, we get
$$\phi(S) = \rho_{S, \sigma}[\rho_{\sigma, \tau}[\phi(\tau)]] = \rho_{S, \sigma}[\phi(\sigma)] \; a.s.$$
In particular,
\begin{equation} \label{fourth}
\rho_{S, \sigma}[\rho_{\sigma, \tau}[\phi(\tau)]] = \rho_{S, \sigma}[\phi(\sigma)] \quad a.s. \,
\end{equation}
Due to the additional assumption, the non-linear operators $\rho$ are strictly monotone. From this, together with equality \eqref{fourth} and the inequality $\rho_{\sigma, \tau}[\phi(\tau)] \leq \phi(\sigma)$ a.s. (which is due to condition 1 of the lemma), we get
$\phi(\sigma) = \rho_{\sigma, \tau}[\phi(\tau)].$
\end{proof}

\begin{lemma}\label{optimality_criterion}
Let $\nu^{*}_{k} \in \Theta_{\theta_{k}}$. We introduce the following two conditions: \\
[0.3cm]
i) $\rho_{\theta_{k}, \nu^{*}_{k}}[V(\nu^{*}_{k})] = \rho_{\theta_{k}, \nu^{*}_{k}}[\xi(\nu^{*}_{k})]$ a.s.\\
[0.2cm]
ii) The family $(V(\nu \wedge \nu^{*}_{k}))_{\nu \in \Theta_{\theta_{k}}}$ is a $(\Theta, \rho)$-martingale family.\\
[0.3cm]
1. (Optimality criterion) If i) and ii) are satisfied, then $\nu^{*}_{k}$ is optimal for problem \eqref{optimal_stopping_problem_kbis}.\\
[0.2cm]
2. If, moreover, the non-linear operator $\rho_{\theta_{k}, \nu^{*}_{k}}$ is assumed to be strictly monotone and satisfies the assumptions of admissibility (ii), knowledge preservation (iii), consistency (v), ``generalized zero-one law'' (vi) and monotone Fatou property (vii), then the converse statement is also true.

\end{lemma}

\begin{Remark}
We note that the property $V(\nu^{*}_{k}) = \xi(\nu^{*}_{k})$ a.s. implies that $\rho_{\theta_{k}, \nu^{*}_{k}}[V(\nu^{*}_{k})] = \rho_{\theta_{k}, \nu^{*}_{k}}[\xi(\nu^{*}_{k})]$ a.s. The converse implication is true under the additional assumption: $\rho_{\theta_{k}, \nu^{*}_{k}}$ is strictly monotone.

\end{Remark}

\begin{proof} First, let us show statement 1.\\
[0.2cm]
Let  $\nu^{*}_{k} \in \Theta_{\theta_{k}}$ be such that the two conditions i) and ii) introduced above are satisfied. By condition ii) the family $(V(\nu \wedge \nu^{*}_{k}))_{\nu\in\Theta_{\theta_{k}}}$ is a $(\Theta, \rho)$-martingale family.\\
[0.2cm]
Hence, for any $\nu \in \Theta_{\theta_{k}}$, we have
$$V(\theta_{k} \wedge \nu^{*}_{k}) = \rho_{\theta_{k}, \nu \wedge \nu^{*}_{k}}[V(\nu \wedge \nu^{*}_{k})] \; a.s.,$$ which implies $$ V(\theta_{k}) = \rho_{\theta_{k}, \nu \wedge \nu^{*}_{k}}[V(\nu \wedge \nu^{*}_{k})] \; a.s.$$
In particular, for $\nu = \nu^{*}_{k}$, we get
$$V(\theta_{k}) = \rho_{\theta_{k}, \nu^{*}_{k} \wedge \nu^{*}_{k}}[V(\nu^{*}_{k} \wedge \nu^{*}_{k})] = \rho_{\theta_{k}, \nu^{*}_{k}}[V(\nu^{*}_{k})] \; a.s.$$
From this, together with  condition i), we have
$$V(\theta_{k}) = \rho_{\theta_{k}, \nu^{*}_{k}}[V(\nu^{*}_{k})] = \rho_{\theta_{k}, \nu^{*}_{k}}[\xi(\nu^{*}_{k})] \; a.s.,$$
which implies that the stopping time $\nu^{*}_{k}$ is an optimal stopping time for problem \eqref{optimal_stopping_problem_kbis}.\\
Now, let us show statement 2. Let $\nu^{*}_{k} \in \Theta_{\theta_{k}}$ be an optimal stopping time for problem \eqref{optimal_stopping_problem_kbis}.
Hence, we have
$$V(\theta_{k}) = \rho_{\theta_{k}, \nu^{*}_{k}}[\xi(\nu^{*}_{k})] \; a.s.$$
By the first part of Theorem \ref{Thm_Snell} (which is applicable as $\rho$ satisfies the assumptions), the value family $V$ is a $(\Theta, \rho)$-supermartingale family.  
Thus, by the $(\Theta, \rho)$-supermartingale property of $V$, and as $\nu^{*}_{k} \in \Theta_{\theta_{k}}$, we have
$$V(\theta_{k}) \geq \rho_{\theta_{k}, \nu^{*}_{k}}[V(\nu^{*}_{k})] \; a.s.$$
On the other hand,  due to the fact that $\xi \leq V$ (cf. Remark \ref{TEXRemark33}, Statement 2) and to the monotonicity of the non-linear operator $\rho_{\theta_{k}, \nu^{*}_{k}}$, it holds 
$$\rho_{\theta_{k}, \nu^{*}_{k}}[\xi(\nu^{*}_{k})] \leq \rho_{\theta_{k}, \nu^{*}_{k}}[V(\nu^{*}_{k})] \; a.s.,$$
Thus, we get
$$V(\theta_{k}) = \rho_{\theta_{k}, \nu^{*}_{k}}[V(\nu^{*}_{k})] = \rho_{\theta_{k}, \nu^{*}_{k}}[\xi(\nu^{*}_{k})] \; a.s.$$
Moreover, since $V(\theta_{k}) = \rho_{\theta_{k}, \nu^{*}_{k}}[V(\nu^{*}_{k})]$, by applying Lemmas \ref{Lemma2525} and \ref{Lemma2626} (the latter is applicable as $\rho$ is assumed to be strictly monotone) with $S = \theta_{k}$, $\tau = \nu^{*}_{k}$, we conclude that $V$ is a $(\Theta, \rho)$-martingale on $[\theta_{k}, \nu^{*}_{k}]$.\\
The proof is complete.
\end{proof}

\subsubsection{Two useful consequences of the DPP}
The following two results hold, if a given admissible $p$-integrable family $\phi$ satisfies  the (DPP) from Eq.\eqref{Eq_def_DPP}, and if $\tilde \nu_k$ is defined by \begin{equation}\label{Eq_nu_tilde} \tilde\nu_k:=
{\rm ess} \inf \;\tilde{\cal A}_{k} \quad \mbox{where} \quad 
\tilde{\cal A}_{k}:=\{\,\tau \in \Theta_{\theta_k}\,: \,\phi({\tau}) = \xi({\tau}) \,\,\mbox {a.s.} \,\}.
\end{equation}
The following lemma is a consequence of the definition of $\tilde\nu_{k}$ and of the DPP.  

\begin{lemma}\label{TEXlemma38}
Assume that $\phi$ satisfies the  DPP holds, and let $\tilde \nu_k$ be defined by \eqref{Eq_nu_tilde}. Then, for each $l \in \{k, k+1, ...\}$, $\phi(\theta_{l}) = \rho_{\theta_{l}, \theta_{l+1}}[\phi(\theta_{l+1})]$ on the set $\{\tilde\nu_{k} > \theta_{l} \}$.

\end{lemma}

\begin{proof}
Let $l \in \{k, k+1, ...\}$. By the definition of $\tilde\nu_{k}$, on the set $\{\tilde\nu_{k} > \theta_{l}\}$, we have $\phi(\theta_{l}) > \xi(\theta_{l})$. From this and from the DPP, we conclude that on the set $\{\tilde\nu_{k} > \theta_{l}\}$,
$\phi(\theta_{l}) = \rho_{\theta_{l}, \theta_{l+1}}[\phi(\theta_{l+1})].$
\end{proof}

\begin{lemma} \label{TEXlemma3939}
Assume that the (DPP) from Eq.\eqref{Eq_def_DPP}  holds. Under the assumptions of (iii) and ``generalized zero-one law'' (vi) on $\rho$, it holds:\\
[0.2cm]
1. For each $l \in \mathbbm{N}$, 
$\phi(\theta_{l} \wedge \tilde\nu_{k}) = \rho_{\theta_{l}, \theta_{l+1} \wedge \tilde\nu_{k}}[\phi(\theta_{l+1} \wedge \tilde\nu_{k})].$\\
[0.2cm]
2. For each $l \in \mathbbm{N}$, 
$\phi(\theta_{l} \wedge \tilde\nu_{k}) = \rho_{\theta_{l} \wedge \tilde\nu_{k}, \theta_{l+1} \wedge \tilde\nu_{k}}[\phi(\theta_{l+1} \wedge \tilde\nu_{k})].$

\end{lemma}

\begin{proof}
First, we show Statement 1 of the Lemma.
\begin{equation}\label{TEXEQUAtiON242424}
\begin{aligned}
\rho_{\theta_{l}, \theta_{l+1} \wedge \tilde\nu_{k}}[\phi(\theta_{l+1} \wedge \tilde\nu_{k})] = 
\mathbbm{1}_{\{\tilde\nu_{k} \leq \theta_{l}\}}\rho_{\theta_{l}, \theta_{l+1} \wedge \tilde\nu_{k}}[\phi(\theta_{l+1} \wedge \tilde\nu_{k})]\\
+ \mathbbm{1}_{\{\tilde\nu_{k} > \theta_{l}\}}\rho_{\theta_{l}, \theta_{l+1} \wedge \tilde\nu_{k}}[\phi(\theta_{l+1} \wedge \tilde\nu_{k})].
\end{aligned}
\end{equation}
For the first summand in Eq.\eqref{TEXEQUAtiON242424}, we note that on the set $\{\tilde\nu_{k} \leq \theta_{l}\}$, $\theta_{l+1}\wedge\tilde\nu_{k} = \theta_{l} \wedge \tilde\nu_{k} = \tilde\nu_{k}$. Hence, by the ``generalized zero-one law'', we have
\begin{equation}\label{TEXEQUATION242424}
\mathbbm{1}_{\{\tilde\nu_{k} \leq \theta_{l}\}}\rho_{\theta_{l}, \theta_{l+1} \wedge \tilde\nu_{k}}[\phi(\theta_{l+1} \wedge \tilde\nu_{k})] = \mathbbm{1}_{\{\tilde\nu_{k} \leq \theta_{l}\}}\rho_{\theta_{l}, \theta_{l} \wedge \tilde\nu_{k}}[\phi(\theta_{l} \wedge \tilde\nu_{k})] = \mathbbm{1}_{\{\tilde\nu_{k} \leq \theta_{l}\}}\phi(\theta_{l} \wedge \tilde\nu_{k}),
\end{equation}
where we have used property (iii) to obtain the last equality (as $\theta_{l} \wedge \tilde\nu_{k} \leq \theta_{l}$).\\
For the second summand in Eq.\eqref{TEXEQUAtiON242424}, we use the ``generalized zero-one law'' to write:
\begin{equation}\label{TEXEQUATION252525}
\mathbbm{1}_{\{\tilde\nu_{k} > \theta_{l}\}}\rho_{\theta_{l}, \theta_{l+1} \wedge \tilde\nu_{k}}[\phi(\theta_{l+1} \wedge \tilde\nu_{k})] = \mathbbm{1}_{\{\tilde\nu_{k} > \theta_{l}\}}\rho_{\theta_{l}, \theta_{l+1}}[\phi(\theta_{l+1})] = \mathbbm{1}_{\{\tilde\nu_{k} > \theta_{l}\}}\phi(\theta_{l}),
\end{equation}
where we have applied Lemma \ref{TEXlemma38} to obtain the last equality.\\
Hence, by replacing Eqs \eqref{TEXEQUATION242424} and \eqref{TEXEQUATION252525} in Eq. \eqref{TEXEQUAtiON242424}, we get
$$\rho_{\theta_{l}, \theta_{l+1} \wedge \tilde\nu_{k}}[\phi(\theta_{l+1} \wedge \tilde\nu_{k})] = \mathbbm{1}_{\{\tilde\nu_{k} \leq \theta_{l}\}}\phi(\theta_{l} \wedge \tilde\nu_{k}) + \mathbbm{1}_{\{\tilde\nu_{k} >\theta_{l}\}}\phi(\theta_{l} \wedge \tilde\nu_{k}) = \phi(\theta_{l} \wedge \tilde\nu_{k}).$$
We now prove Statement 2 of the Lemma.
\begin{equation}\label{TEXEQUAtiON2666}
\begin{aligned}
\rho_{\theta_{l} \wedge \tilde\nu_{k}, \theta_{l+1} \wedge \tilde\nu_{k}}[\phi(\theta_{l+1} \wedge \tilde\nu_{k})] = 
\mathbbm{1}_{\{\tilde\nu_{k} \leq \theta_{l}\}}\rho_{\theta_{l} \wedge \tilde\nu_{k}, \theta_{l+1} \wedge \tilde\nu_{k}}[\phi(\theta_{l+1} \wedge \tilde\nu_{k})]\\
+ \mathbbm{1}_{\{\tilde\nu_{k} >\theta_{l}\}}\rho_{\theta_{l} \wedge \tilde\nu_{k}, \theta_{l+1} \wedge \tilde\nu_{k}}[\phi(\theta_{l+1} \wedge \tilde\nu_{k})].
\end{aligned}
\end{equation}
For the second summand of Eq.\eqref{TEXEQUAtiON2666}, by Statement 1 of the Lemma, we get
\begin{align*}
\mathbbm{1}_{\{\tilde\nu_{k} > \theta_{l}\}}\rho_{\theta_{l} \wedge \tilde\nu_{k}, \theta_{l+1} \wedge \tilde\nu_{k}}[\phi(\theta_{l+1} \wedge \tilde\nu_{k})] &= \mathbbm{1}_{\{\tilde\nu_{k} > \theta_{l}\}}\rho_{\theta_{l}, \theta_{l+1} \wedge \tilde\nu_{k}}[\phi(\theta_{l+1} \wedge \tilde\nu_{k})]\\
&= \mathbbm{1}_{\{\tilde\nu_{k} > \theta_{l}\}}\phi(\theta_{l} \wedge \tilde\nu_{k}).
\end{align*}
For the first summand of Equation \eqref{TEXEQUAtiON2666}, we apply the ``generalized zero-one law'' (as the set $\{\tilde\nu_{k} \leq \theta_{l}\}$ is $\mathcal{F}_{\theta_{l} \wedge \tilde\nu_{k}}$-measurable and on the set $\{\tilde\nu_{k} \leq \theta_{l}\}$, we have $\theta_{l+1} \wedge \tilde\nu_{k} = \theta_{l} \wedge \tilde\nu_{k}$), we have
\begin{align*}
\mathbbm{1}_{\{\tilde\nu_{k} \leq \theta_{l}\}}\rho_{\theta_{l} \wedge \tilde\nu_{k}, \theta_{l+1} \wedge \tilde\nu_{k}}[\phi(\theta_{l+1} \wedge \tilde\nu_{k})] &= \mathbbm{1}_{\{\tilde\nu_{k} \leq \theta_{l}\}}\rho_{\theta_{l} \wedge \tilde\nu_{k}, \theta_{l} \wedge \tilde\nu_{k}}[\phi(\theta_{l} \wedge \tilde\nu_{k})]\\
&= \mathbbm{1}_{\{\tilde\nu_{k} \leq \theta_{l}\}}\phi(\theta_{l} \wedge \tilde\nu_{k}),
\end{align*}
where we have used property (iii) on $\rho$ to obtain the last equality.\\
Finally, we get 
\begin{align*}
\rho_{\theta_{l} \wedge \tilde\nu_{k}, \theta_{l+1} \wedge \tilde\nu_{k}}[\phi(\theta_{l+1} \wedge \tilde\nu_{k})] &= \mathbbm{1}_{\{\tilde\nu_{k} \leq \theta_{l}\}}\phi(\theta_{l} \wedge \tilde\nu_{k}) + \mathbbm{1}_{\{\tilde\nu_{k} \geq \theta_{l + 1}\}}\phi(\theta_{l} \wedge \tilde\nu_{k})\\
&= \phi(\theta_{l} \wedge \tilde\nu_{k}).
\end{align*}
\vspace{-0.5cm}
\end{proof}
\subsubsection{The proof of the optimality of $\nu_k$}
We are now ready to prove Theorem \ref{prop_martingale_optimality_principle_general} on the optimality of the Bermudan stopping time  $\nu_{k}$, defined by \eqref{nu_zerobis}. We will need also the following remark:

\begin{Remark} \label{TEXremark1}
Any admissible  family $(\phi(\tau), \tau \in \Theta)$ in our framework is \textbf{right-continuous along Bermudan stopping times}, that is, for all $\tau \in \Theta$, and for all non-increasing sequences of Bermudan  stopping times $(\tau_{n}) \in \Theta^{\mathbbm{N}}$ such that $\tau_{n} \downarrow \tau$, it holds 
$\lim_{n \to +\infty}\phi(\tau_{n}) = \phi(\tau).$\\
[0.2cm]
Indeed, let $\tau \in \Theta$, and let $(\tau_{n}) \in \Theta^{\mathbbm{N}}$ be such that $\tau_{n} \downarrow \tau$.
For each $n$, we have  
$\tau_{n} = \sum_{l= 0}^{+\infty}\theta_{l}\mathbbm{1}_{A_{l}^{(n)}}+T\mathbbm{1}_{\bar A^{(n)}}$  and   $\tau = \sum_{l= 0}^{+\infty}\mathbbm{1}_{A_{l}}\theta_{l}+T\mathbbm{1}_{\bar A}$ (cf. the canonical writing from Remark \ref{Rk_canonical}).\\
Let $\omega \in \Omega$. Recall that $\{(A_{l})_{l\in\N},\bar A\}$ is a partition of $\Omega$. If $\omega\in\bar A,$ then $\tau(\omega)=T=\tau_n(\omega)$ (as $\tau_n\downarrow\tau$).  Otherwise, there exists a unique $l_{0} = l_{0}(\omega)$, such that $\omega \in A_{l_{0}}$, and $\tau(\omega) = \theta_{l_{0}}(\omega)<T$. Then, as $\tau_{n}(\omega) \downarrow \tau(\omega)$ and as $\theta_{k}(\omega) \uparrow T$, after a certain rank $n_{0} = n_{0}(\omega)$, $\tau_{n}(\omega) = \tau(\omega) = \theta_{l_{0}}(\omega)$.\\
Hence, in both cases, there exists $n_{0} = n_{0}(\omega)$ such that for all $n \geq n_{0}$, $\tau_{n}(\omega) = \tau(\omega)$, and, hence,  by Remark \ref{Rmk_on_admissibility}, for all $n \geq n_{0}$, $\phi(\tau_{n})(\omega) = \phi(\tau)(\omega)$.  We conclude that $\lim_{n \to +\infty}\phi(\tau_{n})(\omega) = \phi(\tau)(\omega).$

\end{Remark}

\begin{proof}[Proof of Theorem \ref{prop_martingale_optimality_principle_general}]
By Lemma \ref{optimality_criterion}, in order to show that $\nu_{k}$, defined in \eqref{nu_zerobis},  is optimal for problem \eqref{optimal_stopping_problem_kbis}, it is enough to show the following two conditions:\\
[0.2cm]
\textbf{(i) } $\rho_{\theta_{k}, \nu_{k}}[V(\nu_{k})] = \rho_{\theta_{k}, \nu_{k}}[\xi(\nu_{k})]$ a.s.;\\
[0.2cm]
\textbf{(ii)} The family $(V(\nu))$ is a $(\Theta, \rho)$-martingale on $[\theta_{k}, \nu_{k}]$.\\
[0.2cm]
We start our proof by showing the second condition first.\\
By Lemma \ref{Lemma_admissible}, $V$ is an admissible family, and it is also $p$-integrable by Assumption \ref{Hyp3}.  By Lemma \ref{Lemma2525} (on the $(\Theta, \rho)$-martingale property), in order to show the second condition, it is enough to show that: for each $\sigma \in \Theta$, such that $\theta_{k} \leq \sigma \leq \nu_{k}$, 
\begin{equation} \label{properTY26}
\rho_{\sigma, \nu_{k}}[V(\nu_{k})] = V(\sigma).
\end{equation}
Let $\sigma \in \Theta_{\theta_{k}}$. Then $\sigma$ is of the form $\sigma = \sum_{m \geq k}\theta_{m}\mathbbm{1}_{A_{m}}+T\mathbbm{1}_{\bar A}$, where $\{(A_{m})_{m\in\N},\bar A\}$ form a partition of $\Omega$; $A_{m}$ is $\mathcal{F}_{\theta_{m}}$-measurable for each $m$, and  $\bar A\in\cf_T.$\\
Hence, to prove Equation \eqref{properTY26}, it is enough to show that, for each $m \geq k$,
\begin{equation}\label{TEXEQuaTIoN2888}
\mathbbm{1}_{A_{m}}\rho_{\theta_{m}, \nu_{k}}[V(\nu_{k})] = \mathbbm{1}_{A_{m}}V(\theta_{m}), \; \; \;
\end{equation}
and 
\begin{equation}\label{NEWTEXEQuaTIoN2888}
\mathbbm{1}_{\bar A}\rho_{T, \nu_{k}}[V(\nu_{k})] = \mathbbm{1}_{\bar A}V(T). \; \; \;
\end{equation}
As $\sigma=\nu_k=T$ on $\bar A$, Eq.\eqref{NEWTEXEQuaTIoN2888} holds true, by the "generalized" zero-one law (vi) and the knowledge preserving property (iii).   \\
Let $m \geq k$ be fixed. The proof of Eq. \eqref{TEXEQuaTIoN2888} passes through the following steps:\\
[0.2cm]
\textbf{$1^{st}$ Step:} As $V$ satisfies the DPP (cf. Theorem \ref{TEXtheorem1}), we use  Lemma \ref{TEXlemma3939}, Statement 2 and the consistency property on $\rho$, to show that, for each fixed $n \in \mathbbm{N}$,
\begin{equation}\label{TEXEQuatIoN2999}
\rho_{\theta_{m} \wedge \nu_{k}, \theta_{m+n} \wedge \nu_{k}}[V(\theta_{m+n} \wedge \nu_{k})] = V(\theta_{m} \wedge \nu_{k}).
\end{equation}
Indeed, by applying successively $(n-m-1)$ times the consistency property on $\rho$ (if $n-m \geq 2$), we get 
\begin{align*}
\rho_{\theta_{m} \wedge \nu_{k}, \theta_{m+n} \wedge \nu_{k}}[V(\theta_{m+n} \wedge \nu_{k})]
= \rho_{\theta_{m} \wedge \nu_{k}, \theta_{m+1} \wedge \nu_{k}}[\rho_{\theta_{m+1} \wedge \nu_{k}, \theta_{m+2} \wedge \nu_{k}}[...\\
[\rho_{\theta_{n+m-1} \wedge \nu_{k}, \theta_{m+n} \wedge \nu_{k}}[V(\theta_{m+n} \wedge \nu_{k})]]]].
\end{align*}
By applying Lemma \ref{TEXlemma3939}, Statement 2, again successively $(n-m)$ times, we get
$$\rho_{\theta_{m} \wedge \nu_{k}, \theta_{m+n} \wedge \nu_{k}}[V(\theta_{m+n} \wedge \nu_{k})] = V(\theta_{m} \wedge \nu_{k}),$$
which proves Equation \eqref{TEXEQuatIoN2999}.\\
[0.2cm] 
Hence, the sequence of random variables $(\rho_{\theta_{m} \wedge \nu_{k}, \theta_{m+n} \wedge \nu_{k}}[V(\theta_{m+n} \wedge \nu_{k})])_{n \in \mathbbm{N}}$ does not depend on $n$ and is constantly equal to the random variable $V(\theta_{m} \wedge \nu_{k})$.\\
[0.2cm] 
\textbf{$2^{nd}$ Step:} As $V$ is left-upper-semicontinuous (LUSC) along the sequence  $(\theta_{m+n}\wedge\nu_{k})_{n\in\N}$ by Assumption \ref{HypA}, and as $\rho$ is LUSC along Bermudan stopping strategies with respect to terminal condition and  terminal time at $(\nu_{k})$, we have
\begin{align*}
\limsup_{n \to +\infty}\rho_{\theta_{m} \wedge \nu_{k}, \theta_{m+n} \wedge \nu_{k}}[V(\theta_{m+n} \wedge \nu_{k})] &\leq \rho_{\theta_{m} \wedge \nu_{k}, \nu_{k}}[\limsup_{n \to +\infty}V(\theta_{m+n} \wedge \nu_{k})]\\
&\leq \rho_{\theta_{m} \wedge \nu_{k}, \nu_{k}}[V(\nu_{k})],
\end{align*}
where we have used the monotonicity of $\rho$ and Assumption \ref{HypA} on $V$ to obtain the last inequality. Hence, 
$$V(\theta_{m} \wedge \nu_{k}) \leq \rho_{\theta_{m} \wedge \nu_{k}, \nu_{k}}[V(\nu_{k})].$$
The opposite inequality holds true due to the $(\Theta, \rho)$-supermartingale property of $V$ (cf. Theorem \ref{Thm_Snell}, Statement 1). Hence, we have the equality, that is, 
$$V(\theta_{m} \wedge \nu_{k}) = \rho_{\theta_{m} \wedge \nu_{k}, \nu_{k}}[V(\nu_{k})],$$ and Eq.\eqref{TEXEQuaTIoN2888} is established.\\   
\textbf{$3^{rd}$ Step:} From the above Eqs \eqref{NEWTEXEQuaTIoN2888} and \eqref{TEXEQuaTIoN2888}, it follows: 
$$\rho_{\sigma \wedge \nu_{k}, \nu_{k}}[V(\nu_{k})] = V(\sigma \wedge \nu_{k}),$$
(as $\sigma = \sum_{m \geq k}\theta_{m}\mathbbm{1}_{A_{m}}+T\mathbbm{1}_{\bar A}$), which proves Eq.\eqref{properTY26}.\\
[0.2cm]
We conclude, by Lemma \ref{Lemma2525}, that $V$ is a $(\Theta, \rho)$-martingale on the stochastic interval $[\theta_{k}, \nu_{k}]$. This shows condition (ii) in the optimality criterion of Lemma \ref{optimality_criterion}.\\
[0.2cm]
It remains for us to show condition (i) in the optimality criterion. Let us recall that $\nu_{k} = \essinf \mathcal{A}_{k}$, where $\mathcal{A}_{k} = \{\tau \in \Theta_{\theta_{k}}: \; V(\tau) = \xi(\tau) \; a.s.\}$. Let $(\tau_{n})$ be a non-increasing sequence in $\mathcal{A}_{k}$, such that $\lim_{n \to +\infty}\downarrow\tau_{n} = \nu_{k}$. As $\tau_{n}$ is in $\mathcal{A}_{k}$, we have $V(\tau_{n}) = \xi(\tau_{n})$. By passing to the limit in this equality and by using that both families $V$ and $\xi$ are right-continuous along the sequence of Bermudan stopping strategies $(\tau_{n})$ (cf. Remark \ref{TEXremark1}), we obtain 
\begin{equation}\label{Eq_equality at nu}
V(\nu_{k}) = \xi(\nu_{k}), 
\end{equation}
which proves condition (i).\\
This concludes the proof of the optimality of $\nu_{k}$.
\end{proof}
\subsubsection{Assumption \ref{HypA} on $V$: Discussion}
Let us now check under which conditions Assumption \ref{HypA} on $V$  holds true.\\
[0.3cm]
Under the assumptions $(a), (b)$ on $\Theta$, the set $\{\omega \in \Omega: \nu_{k}(\omega) = T, \theta_{l}(\omega) < T, \text{for all } l\in\N\}$ might be non-empty. We will show the following lemma.

\begin{lemma}\label{TEXlemmaA2929}
Let $\rho$ satisfy the properties of knowledge preservation (iii), monotonicity (iv), and consistency (v), and the following property
\begin{equation}\label{Eq_property rho}
\limsup_{n \to +\infty}\rho_{\theta_{n}, T}[\eta] \leq \rho_{T, T}[\eta], \text { for all } \eta\in L^p(\cf_T).
\end{equation}
Then,  
\begin{equation}\label{TEXEQuationN277}
\limsup_{n \to +\infty}V(\theta_{n}) \leq \limsup_{n \to +\infty}\xi(\theta_{n}) \vee \xi(T).
\end{equation}

\end{lemma}

\begin{proof}
 
For each $m\in\N$, for each $\tau \in \Theta_{\theta_{m}}$, $\xi(\tau) \leq \sup_{p \geq m}\xi(\theta_{p}) \vee \xi(T)$.\\
[0.2cm]
Indeed, for each $\tau \in \Theta_{\theta_{m}}$, we have
$$\xi(\tau) = \sum_{l \geq m}\xi(\theta_{l})\mathbbm{1}_{A_{l}} + \xi(T)\mathbbm{1}_{\bar{A}} \leq \eta\mathbbm{1}_{\bar{A}^{c}} + \xi(T)\mathbbm{1}_{\bar{A}} \leq \eta \vee \xi(T),$$
where we have set $\eta := \sup_{p \geq m}\xi(\theta_{p})$. 
Let us define $\bar{\eta} \coloneqq \eta \vee \xi(T) = \sup_{p \geq m}\xi(\theta_{p}) \vee \xi(T)$.\\
[0.2cm]
By the monotonicity of $\rho$ and the knowledge preserving property on $\rho$, we have, for all $\tau\in \Theta_{\theta_{m}}, $
\begin{equation}\label{TEXEquationN288}
\xi(\tau)=\rho_{\tau, T}[\xi(\tau)] \leq \rho_{\tau, T}[\eta \vee \xi(T)]=\rho_{\tau, T}[\bar \eta].
\end{equation}
Hence, for each $n \geq m$ (as $\Theta_{n} \subset \Theta_{m}$),
$$V(\theta_{n}) = \esssup_{\tau \in \Theta_{\theta_{n}}}\rho_{\theta_{n}, \tau}[\xi(\tau)] \leq \esssup_{\tau \in \Theta_{\theta_{n}}}\rho_{\theta_{n}, \tau}[\rho_{\tau, T}[\bar{\eta}]]$$
where we have used the monotonicity of $\rho$ and \eqref{TEXEquationN288} for the inequality. By the consistency property on $\rho$, we get
$\rho_{\theta_{n}, \tau}[\rho_{\tau, T}[\bar{\eta}]] = \rho_{\theta_{n}, T}[\bar{\eta}].$\\
Finally, 
$V(\theta_{n}) \leq \esssup_{\tau \in \Theta_{\theta_{n}}}\rho_{\theta_{n}, T}[\bar{\eta}] = \rho_{\theta_{n}, T}[\bar{\eta}].$\\
Hence, 
$\limsup_{n \to +\infty}V(\theta_{n}) \leq \limsup_{n \to +\infty}\rho_{\theta_{n}, T}[\bar{\eta}].$\\
As, by assumption on $\rho$,
$\limsup_{n \to +\infty}\rho_{\theta_{n}, T}[\bar{\eta}] \leq \rho_{T, T}[\bar{\eta}],$
 we obtain 
$$\limsup_{n \to +\infty}V(\theta_{n}) \leq \rho_{T, T}[\bar{\eta}] = \bar{\eta} = \sup_{p \geq m}\xi(\theta_{p}) \vee \xi(T).$$
Hence, by passing to the limit when $m \to +\infty$, we get
$$\limsup_{n \to +\infty}V(\theta_{n}) \leq \limsup_{m \to +\infty}\xi(\theta_{m}) \vee \xi(T),$$
which finishes the proof of the lemma.
\end{proof}

\begin{Assumption}\label{ASSu}
We assume that the pay-off family $\xi$ is LUSC along the sequence $(\theta_n)$ at $T$, that is, $\xi$ is   such that $$\limsup_{n \to +\infty} \xi(\theta_n) \leq \xi(T).$$

\end{Assumption}
\begin{proposition}\label{Prop_final}
If the pay-off family $\xi$ satisfies Assumption \ref{ASSu} (LUSC along the sequence $(\theta_n)$ at $T$), and  if $\rho$ satisfies the properties of admissibility (ii),  knowledge preservation (iii), monotonicity (iv), consistency (v),  "generalized" zero-one law (vi), and LUSC along the sequence $(\theta_n)$ at $T$ (property \eqref{Eq_property rho}), then the value family $V$ satisfies Assumption \ref{HypA}.
\end{proposition}
\begin{proof}
On the set $\{\nu_k<T\}$, we have, for all $n$ after a certain rank $\bar n(\omega)$ (depending on $\omega$), 
$(\theta_{n}\wedge \nu_k)(\omega)=\nu_k(\omega)$ ((due to $\lim\uparrow\theta_{n}=T$). Hence, by Remark \ref{Rmk_on_admissibility}), for all $n\geq n_0(\omega)$, $V(\theta_{n}\wedge \nu_k)(\omega)=V(\nu_k)(\omega)$
 Hence, 
$$\mathbbm{1}_{\{\nu_k<T\}}\limsup_{n \to +\infty} V(\theta_n\wedge \nu_k)=\mathbbm{1}_{\{\nu_k<T\}}\lim_{n \to +\infty} V(\theta_n\wedge \nu_k)=\mathbbm{1}_{\{\nu_k<T\}} V( \nu_k).$$
On the other hand, under the Assumption \ref{ASSu} on $\xi$, we have, by Lemma \ref{TEXlemmaA2929}, 
$$\limsup_{n \to +\infty}V(\theta_{n}) \leq \limsup_{n \to +\infty}\xi(\theta_{n})\vee\xi(T) \leq \xi(T) = V(T).$$
Hence, using the admissibility of $V$, $$\mathbbm{1}_{\{\nu_k=T\}}\limsup_{n \to +\infty}V(\theta_{n}\wedge \nu_k)=\mathbbm{1}_{\{\nu_k=T\}}\limsup_{n \to +\infty}V(\theta_{n})\leq \mathbbm{1}_{\{\nu_k=T\}}V(T)=\mathbbm{1}_{\{\nu_k=T\}}V(\nu_k)$$
So, the desired LUSC property at $\nu_k$ on $V$ (Assumption \ref{HypA})  holds true.
\end{proof}
\section{Examples} \label{SectioN3}
In this section we provide some examples of non-linear operators $\rho$, known from the stochastic control and mathematical finance literature, which enter into our framework.
\subsection{Non-linear operators induced by BSDEs}
In this example, $p = 2$.
\subsubsection{The $g$-evaluations}
\cite{Pe04} and \cite{EQ} introduced a type of non-linear evaluation, now known as $g$-evaluation, via a non-linear backward stochastic differential equation (BSDE) with a driver $g$.\\
Let $T > 0$ be a fixed time horizon. We place ourselves in the Brownian framework (for simplicity). Let $(\Omega, \mathcal{F},  \mathbbm{P})$ be a complete  probability space, endowed with a $d$-dimensional Brownian motion $(W_{t})_{t \in [0, T]}$, and let $(\mathcal{F}_{t})_{t \in [0, T]}$ be  the (augmented) natural filtration of the Brownian motion. \\
[0.2cm]
Let $g = g(\omega, t, y, z): \Omega \times [0, T] \times \mathbbm{R} \times \mathbbm{R}^{d} \to \mathbbm{R}$ be \emph{Lipschitz driver}, that is, a function satisfying the following conditions:\\
[0.2cm]
\ding{192} For each $y \in \mathbbm{R}, z \in \mathbbm{R}^{d}$, $g(\cdot, \cdot, y, z) \in L^{2}(\Omega \times [0, T])$ and $g$ is progressively measurable;\\
[0.2cm]
\ding{193} There exists $C > 0$ such that for each $y_{1}, y_{2} \in \mathbbm{R}$, and for each $z_{1}, z_{2} \in \mathbbm{R}^{d}$, $|g(\omega, t, y_{1}, z_{1}) - g(\omega, t, y_{2}, z_{2})| \leq C(|y_{1} - y_{2}| + \|z_{1} - z_{2}\|)$, uniformly for a.e. $(\omega. t)$, where $\|\cdot\|$ denotes the Euclidean norm on $\mathbbm{R}^{d}$;\\
[0.2cm]
Let us  consider the following 1-dimensional BSDE with terminal time $t$ and terminal condition $\eta$, defined on the interval $[0, t]$, given that $0 \leq t \leq T$ and $\eta \in L^{2}(\mathcal{F}_{t})$:
$$y_{s} = \eta + \int^{t}_{s}g(r, y_{r}, z_{r})dr - \int^{t}_{s}z_{r}dB_{r}, \; \; \; s \in [0, t].$$
\begin{definition}{(g-evaluation)}
For each $0 \leq s \leq t \leq T$ and $\eta \in L^{2}(\mathcal{F}_{t})$, we define
$$\mathcal{E}^{g}_{s, t}[\eta] \coloneqq y_{s}.$$
The family of operators $\mathcal{E}^{g}_{s, t}[\cdot]: L^{2}(\mathcal{F}_{t}) \to L^{2}(\mathcal{F}_{s}), 0 \leq s \leq t \leq T$ are called $g$-evaluation.
\end{definition}
\noindent
We recall (cf. \cite{ElKaroui}) that if the terminal time is given by a stopping time $\tau$ valued in $[0,T]$ and if $\eta$ is $\mathcal{F}_{\tau}$-measurable, the solution of the BSDE  with terminal time $\tau$, terminal condition $\eta$ and Lipschitz driver $g$ is defined as the solution of the BSDE with (deterministic) terminal time $T$, terminal condition $\eta$ and Lipschitz driver $g^{\tau}$ defined by $g^{\tau}(t, y, z) \coloneqq g(t, y, z)\mathbbm{1}_{\{t \leq \tau\}}.$
The first component of this solution at time $t$ is  equal to $\mathcal{E}^{g^{\tau}}_{t, T}(\eta)$, also denoted by $\mathcal{E}^{g}_{t, \tau}(\eta)$. We have $\mathcal{E}_{t, \tau}^{g}(\eta) = \eta$ a.s. on the set $\{t \geq \tau\}$. \\
The following result summarizes some of the well-known properties of the $g$-evaluations. 

\begin{proposition}\label{TEXtheoremM31}
Let $g$ satisfy \ding{192} and \ding{193}. Let $S, \tau, \theta$ be stopping times. Then the $g$-evaluation satisfies the following properties:\\
[0.2cm]
(A1) (Monotonicity) $\mathcal{E}^{g}_{S, \tau}[\eta] \leq \mathcal{E}^{g}_{S, \tau}[\eta']$, if $\eta \leq \eta'$;\\
[0.2cm]
(A2) (Knowledge preserving) $\mathcal{E}^{g}_{\tau, S}[\eta] = \eta$, for all $S, \tau$, such that $S \leq \tau$, for all $\eta \in L^{2}(\mathcal{F}_{S})$. \\
[0.2cm]
(A3) (Time consistency) $\mathcal{E}^{g}_{S, \theta}[\mathcal{E}^{g}_{\theta, \tau}[\eta]] = \mathcal{E}^{g}_{S, \tau}[\eta]$, for all $S \leq \theta \leq \tau$, for all  $\eta \in L^{2}(\mathcal{F}_{\tau})$;\\
[0.2cm]
(A4) ("Generalized" zero-one law) 
$\; I_A\rho_{S,\tau}[\xi(\tau)] =I_A\rho_{S,\tau'}[\xi(\tau')],$ for all $A\in\cf_S$, $\tau\in \Theta_{S}$, $\tau'\in\Theta_{S}$ such that $\tau=\tau'$ on $A$.  \\
[0.2cm]
(A5) (Continuity with respect to terminal time and terminal condition)\\
Let $(\tau_{n})_{n \in \mathbbm{N}}$ be a sequence of stopping times in $\mathcal{T}_{S, \tau}$, such that $\lim_{n \to \infty}\tau_{n} = \tau$ a.s. Let $(\eta_{n})_{n \in \mathbbm{N}}$ be a sequence of random variables, such that $\eta_{n} \in L^{2}(\mathcal{F}_{\tau_{n}})$, $\sup_{n}\eta_{n} \in L^{2}$ and $\lim_{n \to \infty}\eta_{n} = \eta$ a.s. Then, we have $\lim_{n \to \infty}\mathcal{E}^{g}_{S, \tau_{n}}[\eta_{n}] = \mathcal{E}^{g}_{S, \tau}[\eta]$ a.s.

\end{proposition}

\begin{remark}
For Property (A4) we refer, e.g., to \cite{Grigorova-2}. Property (A5) was proven in \cite{Quenez-Sulem-0} (in the case of jumps). In \cite{Pe04}, the additional assumption $g(\cdot, 0, 0) = 0$ is made ensuring that the $g$-evaluation satisfies the \textbf{usual} ``zero-one law'': for each $S \leq \tau$, for each $A\in \mathcal{F}_{S},$  $\mathcal{E}^{g}_{S, \tau}[\mathbbm{1}_{A}\eta] = \mathbbm{1}_{A}\mathcal{E}^{g}_{S, \tau}[\eta]$.

\end{remark}
\noindent
Moreover,  the $g$-evaluation $\mathcal{E}^{g}$ satisfies the property  \eqref{cond_on_rho} in Theorem \ref{prop_martingale_optimality_principle_general}. Indeed, we have
$\limsup_{n \to +\infty}\mathcal{E}^{g}_{S, \tau_{n}}[\phi(\tau_{n})] \leq \mathcal{E}^{g}_{S, \tau^{\star}}[\limsup_{n \to +\infty}\phi(\tau_{n})],$
for each non-decreasing sequence $\tau_{n}$ such that $\lim_{n \to +\infty}\tau_{n} = \tau^{\star}$, and for each square-integrable admissible family $\phi$, such that  $\sup_n |\phi(\tau_n)|\in L^2.$ For a proof of this property, based on property (A5) of the $g$-evaluations,  we refer to Lemma A.5 in \cite{Dumitrescu-Quenezz-Sulem}. \\
[0.2cm]
Moreover, in the Brownian framework, the first component $(y_{t})$ of the solution of the BSDE with Lipschitz driver $g$, terminal time $T$, and terminal condition $\eta\in L^2(\cf_T)$,  has continuous trajectories (in $t$). Hence, for any non-decreasing sequence $(\tau_{n})$, such that $\lim_{n \to +\infty}\tau_{n} = T$, 
$$\lim_{n \to +\infty}\mathcal{E}^{g}_{\tau_{n}, T}(\eta)=\lim_{n \to +\infty} y_{\tau_{n}} =  y_{T} = \mathcal{E}^{g}_{T, T}(\eta) = \eta.$$
Thus,  property \eqref{Eq_property rho} from Lemma \ref{TEXlemmaA2929} and Proposition 
\ref{Prop_final}  is satisfied.\\
Hence, the $g$-evaluation satisfies all the properties of the non-linear operators $\rho$ used in our results. \\
 It remains for us to argue that the integrability Assumption \ref{Hyp3} (with $p=2$) on the value family $V$ is satisfied, under some suitable assumptions on $\xi$. The following remark is an application of Remark \ref{TEXRemaRk24} to the framework of a complete financial market model with \textit{imperfections encoded in the driver} $g$ of the dynamics of self-financing portfolios.

\begin{remark}
Let us place ourselves in a complete financial market model with possible imperfections (such as e.g. trading constraints, different interest rates for borrowing and lending, different repo rates, etc.). Let $g$ be the driver from the dynamics of self-financing portfolios in this market. If the family $\xi=(\xi(\tau))_{\tau \in \Theta}$ is assumed to be square-integrable and super-replicable by a self-financing portfolio with wealth $X(\tau)$ at time $\tau$, then, the family $(X(\tau))_{\tau \in \Theta}$ is a $(\Theta, \mathcal{E}^{g})$-martingale. By Remark \ref{TEXRemaRk24} the value $V = (V(\tau))_{\tau \in \Theta}$ is square-integrable (that is, satisfies Assumption \ref{Hyp3} with $p=2$).

\end{remark}

\subsubsection{Peng's $g$-expectation}
Peng's $g$-expectation is a particular case of the previous example, introduced in \cite{Peng-2}. In this case, the driver $g$ is assumed to satisfy the conditions \ding{192}, \ding{193}, and condition \ding{194}' ($g(\cdot, y, 0) \equiv 0$, for all $y \in \mathbbm{R}$). In this particular case, the non-linear operators do not depend on the second index, but on the first index only. Moreover, they satisfy the usual zero-one law. \\
More precisely, let 
 $g$ satisfy conditions \ding{192}, \ding{193} and \ding{194}'. The operators $\mathcal{E}_{g}[\cdot]$ and $\mathcal{E}_{g}[\cdot|\mathcal{F}_{S}]$ are defined by
$\mathcal{E}_{g}[\cdot] \coloneqq \mathcal{E}^{g}_{0, \tau}[\cdot], \; \; \mathcal{E}_{g}[\cdot|\mathcal{F}_{S}] \coloneqq \mathcal{E}^{g}_{S, \tau}[\cdot].$
The family of non-linear operators $(\mathcal{E}_{g}[\cdot|\mathcal{F}_{S}])$, indexed by the stopping times $S$, is called $g$-expectation.

 The $g$-expectation $(\mathcal{E}_{g}[\cdot|\mathcal{F}_{S}])$ satisfies all the properties of the $g$-evaluation and additionally the following property:\\
 (\textbf{usual} zero-one law) For  stopping times $S$, $\tau$ such that  $S \leq \tau$, for  $A\in \mathcal{F}_{S}$, $\mathcal{E}_{g}[\mathbbm{1}_{A}\eta|\mathcal{F}_{S}] = \mathbbm{1}_{A}\mathcal{E}_{g}[\eta|\mathcal{F}_{S}]$.

\subsection{Bayraktar - Yao's non-linear expectations}

Let $T > 0$ and let $(\Omega, \mathcal{F}, \mathbbm{P})$ be a complete probability space endowed with a filtration $(\mathcal{F}_{t})_{t \in [0, T]}$, satisfying the usual conditions.\\
 A non-linear expectation, called $F$-expectation, depending on one time index only,  was introduced in the work by \cite{Bayraktar}, and  optimal stopping problems in continuous time with $F$-expectations were studied in \cite{Bayraktar-3}.
For simplicity of the exposition, we will consider the case where the domain of the $F$-expectation is the whole space $L^\infty$. The non-linear operators ($F$-expectation) in \cite{Bayraktar} are defined first for deterministic times, then extended to stopping times valued in a finite deterministic grid, then, extended  to general stopping times valued in $[0,T].$ We will not repeat the construction here, but will recall the basic properties only (cf.  \cite{Bayraktar} for explanations and details; in particular, cf. Propositions 2.7, 2.8, and 2.9 therein\footnote{Note that in the case of the present exposition, $Dom^{\#}(\mathcal{E})$ of  \cite{Bayraktar} is equal to $L^\infty$.}).  

We recall that, an $F$-expectation,  is a family of operators $\mathcal{E}:=(\mathcal{E}[\cdot|\mathcal{F}_{S}])$, such that $\mathcal{E}[\cdot|\mathcal{F}_{S}]: L^{\infty}(\Omega, \mathcal{F}_{T}, \mathbbm{P}) \to L^{\infty}(\Omega, \mathcal{F}_{S}, \mathbbm{P})$, satisfying the following properties, for all  $S$,$\tau$ stopping times valued in $[0,T]$:\\
[0.2cm]
(C1) (Monotonicity and Positively strict monotonicity) $\mathcal{E}[\eta|\mathcal{F}_{S}] \leq \mathcal{E}[\eta'|\mathcal{F}_{S}]$, if $\eta \leq \eta'$. Moreover, if $0 \leq \eta \leq \eta'$  and $\mathcal{E}[\eta|\mathcal{F}_{0}] = \mathcal{E}[\eta'|\mathcal{F}_{0}]$, then $\eta = \eta'$;\\
[0.2cm]
(C2) (Consistency) $\mathcal{E}[\mathcal{E}[\eta|\mathcal{F}_{\tau}]|\mathcal{F}_{S}] = \mathcal{E}[\eta|\mathcal{F}_{S}]$, for any stopping times $S,\tau$ such  that $0 \leq S \leq \tau \leq T.$\\
[0.2cm]
(C3) (Usual zero-one law) $\mathcal{E}[\mathbbm{1}_{A}\eta|\mathcal{F}_{S}] = \mathbbm{1}_{A}\mathcal{E}[\eta|\mathcal{F}_{S}]$, for any $A \in \mathcal{F}_{S}$.\\
[0.2cm]
(C4) (Translation invariance) $\mathcal{E}[ \eta+X|\mathcal{F}_{S}] = \mathcal{E}[\eta|\mathcal{F}_{S}] + X$, if $X \in L^{\infty}(\mathcal{F}_{S})$.\\
[0.2cm]
(C5) (Knowledge preservation) 
$\mathcal{E}[\eta|\mathcal{F}_{S}] = \eta$, if $\eta \in L^{\infty}(\mathcal{F}_{S})$.\\
[0.2cm]
(C6) (Local property) $\mathcal{E}[\eta \mathbbm{1}_{A} + \eta' \mathbbm{1}_{A^{c}}|\mathcal{F}_{S}] = \mathcal{E}[\eta\mathbbm{1}_{A}|\mathcal{F}_{S}] + \mathcal{E}[\eta'\mathbbm{1}_{A^{c}}|\mathcal{F}_{S}]$, for any $A \in \mathcal{F}_{S}$ and for any $\eta, \xi \in L^{\infty}(\mathcal{F}_{T}).$\\
[0.2cm]
(C7) (Fatou property)
Let $(\eta_n)$ be a sequence in $L^\infty$,  satisfying  $\inf_n \eta_n \geq c$, a.s. for some constant $c\in\R$, and such that  $\lim_{n \to \infty}\eta_{n} = \eta$, where $\eta\in L^\infty.$  Then, 
$\mathcal{E}[\eta|\mathcal{F}_{S}]\leq \liminf_{n\to+\infty} \mathcal{E}[\eta_n|\mathcal{F}_{S}]. $\\
[0.2cm]
(C8) (Dominated convergence) 
Let $(\eta_{n})$ be a sequence, such that $\inf_n \eta_n \geq c$, a.s. for some constant $c\in\R$, and such that $\lim_{n \to \infty}\eta_{n} = \eta$. If there exists  $\bar \eta \in L^{\infty}$ such that $\eta_{n} \leq \bar \eta$  for all $n \in \mathbbm{N}$, then the limit $\eta \in L^{\infty}$, and 
$\lim_{n \to \infty}\mathcal{E}[\eta_{n}|\mathcal{F}_{S}] = \mathcal{E}[\eta|\mathcal{F}_{S}].$\\

Note that the knowledge preserving property is called constant preserving in \cite{Bayraktar}.
\begin{remark}
The $F$-expectation satisfies property \eqref{cond_on_rho} of Theorem \ref{prop_martingale_optimality_principle_general}. 
Indeed, let $(\tau_{n})$ be a non-decreasing sequence of stopping times, such that $\lim_{n \to +\infty} \uparrow \tau_{n} = \tau$. Let $\phi$ be an  admissible $L^{\infty}$-integrable family such that $\sup_{n\in\N} |\phi(\tau_{n})|\in L^{\infty}$. As the $F$-expectation does not depend on the second index, to show property \eqref{cond_on_rho}, we need to show 
$$\limsup_{n \to +\infty} \mathcal{E}[\phi(\tau_{n})|\mathcal{F}_{S}] \leq \mathcal{E}[\limsup_{n \to +\infty}\phi(\tau_{n})|\mathcal{F}_{S}].$$
For each $n \in \mathbbm{N}$, $\phi(\tau_{n}) \leq \sup_{p \geq n}\phi(\tau_{p}).$
Hence, by monotonicity of $\mathcal{E}[\cdot|\mathcal{F}_{S}]$, we have
$$\mathcal{E}[\phi(\tau_{n})|\mathcal{F}_{S}] \leq \mathcal{E}[\sup_{p \geq n}\phi(\tau_{p})|\mathcal{F}_{S}].$$
The sequence $\eta_{n}$ defined by $\eta_{n} = \sup_{p \geq n}\phi(\tau_{p})$ is non-increasing and tends to $\eta \coloneqq \limsup_{n \to +\infty}\phi(\tau_{n})$. Moreover, $\sup_{p \in \mathbbm{N}}|\phi(\tau_{p})| \in L^{\infty}$.   Hence, by property (C8), we have 
$\lim_{n \to +\infty}\mathcal{E}[\eta_{n}|\mathcal{F}_{S}] = \mathcal{E}[\lim_{n \to +\infty}\eta_{n}|\mathcal{F}_{S}] = \mathcal{E}[\eta|\mathcal{F}_{S}].$
Hence,
$$\limsup_{n \to +\infty}\mathcal{E}[\phi(\tau_{n})|\mathcal{F}_{S}] \leq \lim_{n \to +\infty}\mathcal{E}[\eta_{n}|\mathcal{F}_{S}] = \mathcal{E}[\eta|\mathcal{F}_{S}] = \mathcal{E}[\limsup_{n \to +\infty}\phi(\tau_{n})|\mathcal{F}_{S}],$$
which is the desired property.

\end{remark}
\noindent
The assumptions imposed on the pay-off in Bayraktar - Yao (\cite{Bayraktar} and \cite{Bayraktar-3}) ensure that the value in their case belongs to the domain of the operator. 
Hence, if we consider the Bermudan style version of their problem, Assumption \ref{Hyp3} on $V$ will also be satisfied.

\subsection{Dynamic Concave Utilities}
The dynamic concave utilities are among the examples of non-linear operators depending on two time indices. \\
In this example the space is $L^{\infty}$ (that is $p = +\infty$).\\
We place ourselves again in the Brownian framework. A representation result, with an explicit form for the penalty term, for dynamic concave utilities was established in \cite{Delbaen}. The optimal stopping problem with dynamic concave utilities was studied by Bayraktar, Karatzas and Yao in \cite{Bayraktar-2}, where the authors rely on the representation result from \cite{Delbaen}.\footnote{In \cite{Bayraktar-2} the authors choose a different sign convention on $\rho$, and hence, study a minimization optimal stopping problem with dynamic convex risk measures.}\\ 
We recall the following definition from \cite{Delbaen}.  

\begin{definition}{(Dynamic concave utility)}
For  $S, \tau \in \mathcal{T}_{0, T}$, such that $S \leq \tau$, let  
$$\{u_{S, \tau}(\cdot): L^{\infty}(\mathcal{F}_{\tau}) \to L^{\infty}(\mathcal{F}_{S}) \}$$
be  a family of operators.
This family is called a dynamic concave utility, if it satisfies the following properties:\\
[0.2cm]
(D1) (Monotonicity) $u_{S, \tau}(\eta) \leq u_{S, \tau}(\eta')$, if $\eta \leq \eta'$;\\
[0.2cm]
(D2) (Translation invariance) $u_{S, \tau}(\eta + X) = u_{S, \tau}(\eta) + X$, if $\eta \in L^{\infty}(\mathcal{F}_{\tau})$ and $X \in L^{\infty}(\mathcal{F}_{S})$;\\
[0.2cm]
(D3) (Concavity) $u_{S, \tau}(\lambda \eta + (1 - \lambda) \eta') \geq \lambda u_{S, \tau}(\eta) + (1 - \lambda)u_{S, \tau}(\eta')$, for any $\lambda \in [0, 1]$ and $\eta, \eta' \in L^{\infty}(\mathcal{F}_{\tau})$;\\
[0.2cm]
(D4) (Normalisation) $u_{S, \tau}(0) = 0$.

\end{definition}
\noindent
Moreover, in \cite{Delbaen} and \cite{Bayraktar-2} the following properties on the dynamic concave utilities are assumed:\\
[0.2cm]
(D5) (Time consistency) for any stopping time $\sigma \in \mathcal{T}_{S, \tau}$, we have $u_{S, \sigma}(u_{\sigma, \tau}(\eta)) = u_{S, \tau}(\eta)$;\\
[0.2cm]
(D6) (Continuity from above) for any non-increasing sequence $(\eta_{n}) \subset L^{\infty}(\mathcal{F}_{\tau})$ with $\eta = \lim_{n \to \infty} \downarrow \eta_{n} \in L^{\infty}(\mathcal{F}_{\tau})$, we have $\lim_{n \to \infty} \downarrow u_{S, \tau}(\eta_{n}) = u_{S, \tau}(\eta)$\\
[0.2cm]
(D7) (Local property) $u_{S, \tau}(\eta \mathbbm{1}_{A} + \xi \mathbbm{1}_{A^{c}}) = u_{S, \tau}(\eta)\mathbbm{1}_{A} + u_{S, \tau}(\xi)\mathbbm{1}_{A^{c}}$, for any $A \in \mathcal{F}_{S}$ and for any $\eta, \xi \in L^{\infty}(\mathcal{F}_{\tau})$;\\
[0.2cm]
(D8) $E_{P}[\eta|\mathcal{F}_{t}] \geq 0$ for any $\eta \in L^{\infty}(\mathcal{F}_{T})$, such that $u_{t, T}(\eta) \geq 0$.

\begin{remark}
Note that the assumptions in (D6) imply that, for each $n \in \mathbbm{N}$, $\eta_{0} \leq \eta_{n} \leq \eta$, where $\eta_{0} \in L^{\infty}$ and $\eta \in L^{\infty}$. Hence, $\sup_{n} \eta_{n} \in L^{\infty}$ in (D6).  

\end{remark}
By the results of \cite{Delbaen}, any functional $\rho$ satisfying properties (D1) - (D8), has the following representation:
\begin{equation} \label{TEXequation1}
\begin{aligned}
u_{S, \tau}(\eta) &=\essinf_{Q: Q\sim P,Q=P \text{ on } \cf_S  } \Big\{E_{Q}[\eta|\mathcal{F}_{S}]+c_{S,\tau}(Q)  \Big\}\\
&= \essinf_{Q \in \mathcal{Q}_{S}} E_{Q}[\eta + \int^{\tau}_{S}f(u, \psi_{u}^{Q})du|\mathcal{F}_{S}],
\end{aligned}
\end{equation}
where the function $f$ is  such that $f(\cdot,\cdot,x)$ is predictable for any $x$; $f$ is a proper, convex function in the space variable $x$, and valued in $[0,+\infty]$, and the process $(\psi_{t}^{Q})$ is the process from the Doleans-Dade exponential representation for the density process $(Z^{Q}_{t})$, where
$Z_{t}^{Q} = \frac{dQ}{dP}|_{\mathcal{F}_{t}},$
and 
$$\mathcal{Q}_{S} = \{Q: Q \sim P, \psi^{Q}_{t}(\omega) = 0, dt \otimes dP \text{ a.e. on } [[0, S[[, \; E_{Q}[\int^{T}_{S}f(s, \psi^{Q}_{s})ds] < +\infty \}.$$
\noindent
\begin{remark}
A close inspection of the proof of the duality result in (\cite{Biagini} and \cite{Delbaen}) reveals that the dynamic concave utilities depend on the second index only via their penalty term.
\end{remark}
It has been noted in \cite{{Delbaen}} that property (D8) is equivalent to 
\begin{equation}\label{Eq_D8}
c_{t,T}(P)=0, \text{ for all } t\in[0,T].
\end{equation}

\noindent
The dynamic concave utilities $u_{S, \tau}$ do not enter directly into the framework of the present paper, as they are defined only for $S, \tau$ such that $S \leq \tau$ a.s. (cf. \cite{Delbaen} and \cite{Bayraktar-2}). There is, however, a ``natural'' extension of $u_{S, \tau}$ in view of the representation property \eqref{TEXequation1}. This extension is as follows:\\
[0.2cm]
For $S$ and $\tau$ stopping times, and $\eta \in L^{\infty}(\mathcal{F}_{\tau})$, we define
\begin{equation} \label{TexEquatiON1}
\mathbbm{1}_{\{S > \tau\}}u_{S, \tau}(\eta) = \mathbbm{1}_{\{S > \tau\}} \times \eta.
\end{equation}
\begin{remark} 
By properties (D4) and (D7) of the dynamic concave utilities, we get that the ``generalized zero-one law'' is satisfied. 
Indeed, let $A \in \mathcal{F}_{S}$, and let $\tau, \tau' \in \Theta_{S}$ be such that $\tau = \tau'$ on $A$. Let $\eta \in L^{\infty}(\mathcal{F}_{\tau})$. By applying property (D7) with $\xi = 0$, we get
$u_{S, \tau}(\mathbbm{1}_A\eta) = \mathbbm{1}_{A}u_{S, \tau}(\eta) + \mathbbm{1}_{A^{c}}u_{S, \tau}(0).$
As $u_{S, \tau}(0) = 0$ due to the normalisation property (D4), we obtain
$u_{S, \tau}(\mathbbm{1}_A\eta) = \mathbbm{1}_{A}u_{S, \tau}(\eta).$
Hence, the ``generalized zero-one law'' is satisfied.

\end{remark}

\begin{remark}
Property translation invariance (D2)  and property normalisation (D4)  imply that $u_{S, \tau}(\cdot)$ satisfies the knowledge preserving property   (iii) of $\rho$. Indeed, for any $\mathcal{F}_{S}$-measurable $\eta$, we have, by (D2) and (D4), 
$u_{S, \tau}(\eta) = u_{S, \tau}(0) + \eta = 0 + \eta = \eta,$
which shows property (iii).

\end{remark}

\begin{remark}
The dynamic concave utilities satisfy   property  \eqref{cond_on_rho} in Theorem \ref{prop_martingale_optimality_principle_general}.
Indeed, let $(\tau_n)$ be a non-decreasing sequence of stopping times   such that  $\tau_n\uparrow \tau$, and let $\phi$ be an $L^\infty$-integrable admissible family such that  $\sup_n |\phi(\tau_n)|\in L^\infty$. Then, for each $n\in\N$,  
\begin{equation}\label{Eq_DCU1}
\phi(\tau_n)+\int^{\tau_n}_{S}f(u, \psi_{u}^{Q})du\leq
\sup_{p\geq n}\phi(\tau_p)+\int^{\tau}_{S}f(u, \psi_{u}^{Q})du=:\eta_n,
\end{equation}
where, for the inequality,  we have used that $f$ is valued in $[0,+\infty]$, and $\tau_n\leq \tau.$ As 
$\essinf_{Q \in \mathcal{Q}_{S}} E_{Q}[\cdot|\cf_S]$ is non-decreasing, we get: for each $n\in\N$, 
\begin{equation}\label{Eq_DCU3}
\essinf_{Q \in \mathcal{Q}_{S}} E_{Q}[\phi(\tau_n)+\int^{\tau_n}_{S}f(u, \psi_{u}^{Q})du|\cf_S]\leq 
\essinf_{Q \in \mathcal{Q}_{S}} E_{Q}[\eta_n|\cf_S]. 
\end{equation}
We have $\eta_n\downarrow \eta$, where $\eta:=\limsup_{n\to+\infty}\phi(\tau_n)+\int^{\tau}_{S}f(u, \psi_{u}^{Q})du. $ \\As 
$\essinf_{Q \in \mathcal{Q}_{S}} E_{Q}[\cdot|\cf_S]$ is continuous from above, we deduce
\begin{equation}\label{Eq_DCU2}
\lim_{n\to+\infty}\essinf_{Q \in \mathcal{Q}_{S}} E_{Q}[\eta_n|\cf_S]=\essinf_{Q \in \mathcal{Q}_{S}} E_{Q}[\eta|\cf_S].
\end{equation}
Hence, from Eqs.  \eqref{Eq_DCU3} and \eqref{Eq_DCU2},  we get 
$$\limsup_{n\to+\infty}\essinf_{Q \in \mathcal{Q}_{S}} E_{Q}[\phi(\tau_n)+\int^{\tau_n}_{S}f(u, \psi_{u}^{Q})du|\cf_S]\leq  \lim_{n\to+\infty}\essinf_{Q \in \mathcal{Q}_{S}} E_{Q}[\eta_n|\cf_S]=\essinf_{Q \in \mathcal{Q}_{S}} E_{Q}[\eta|\cf_S],$$
which, by the representation result \eqref{TEXequation1}, gives
$$\limsup_{n\to+\infty}u_{S,\tau_n}(\phi(\tau_n))\leq u_{S,\tau_n}(\limsup_{n\to+\infty}\phi(\tau_n))$$.  This  is the desired property \eqref{cond_on_rho}. 
\end{remark}

\begin{remark} The dynamic concave utilities satisfy property \eqref{Eq_property rho} from Lemma \ref{TEXlemmaA2929} and Proposition 
\ref{Prop_final}. Indeed, for $\tau_n\uparrow T$, 
\begin{equation}
\begin{aligned}
u_{\tau_n,T}(\eta)&=\essinf_{Q: Q\sim P,Q=P \text{ on } \cf_{{\tau_n}}} \Big\{E_{Q}[\eta|\mathcal{F}_{\tau_n}]+c_{\tau_n,T}(Q)  \Big\}\leq E_{P}[\eta|\mathcal{F}_{\tau_n}]+c_{\tau_n,T}(P)\\
&=E_{P}[\eta|\mathcal{F}_{\tau_n}],  \\ 
\end{aligned}
\end{equation}
where we have used that for each $n\in\N$, $c_{\tau_n,T}(P)=0$ by Eq.\eqref{Eq_D8}. Hence,  $u_{\tau_n,T}(\eta)\leq E_{P}[\eta|\mathcal{F}_{\tau_n}].$ The sequence $( E_{P}[\eta|\mathcal{F}_{\tau_n}])$ being a uniformly integrable $P$-martingale, with terminal value $E_{P}[\eta|\mathcal{F}_{T}]=\eta$, we get
$$ \limsup_{n\to\infty} u_{\tau_n,T}(\eta)\leq \lim_{n\to\infty} E_{P}[\eta|\mathcal{F}_{\tau_n}]=\eta,$$
which is the desired property. 
\end{remark}
\noindent

\noindent
To finish, the pay-off process $(\xi_{t})$ in \cite{Bayraktar-2} is assumed to be bounded. Hence, by the monotonicity and the knowledge preservation  of $u$, if we consider the Bermudan-style version of the problem studied in \cite{Bayraktar-2}, then the value $V$ satisfies the integrability Assumption \ref{Hyp3} (that is, for each $S \in \Theta, V(S) \in L^{\infty}$).

\section[Appendix]{Appendix: The case of a finite number of pre-described stopping times}

In this appendix, we treat the particular case where  $(\theta_k)_{k\in\N_0}$ is constant from a certain term, \textbf{\textit{independent of }$\omega$}, onwards. More precisely, we place ourselves in the situation where there exists $n\in\N*$ (independent of $\omega$) such that for each $m\geq n$, $\theta_m=T$ 
\begin{theorem}\label{thm_martingale_arretee} 
Let $\phi=(\phi(\tau), \, \tau \in \Theta)$  be a $p$-integrable admissible family. Under the assumptions of knowledge preservation (iii) and ``generalized zero-one law'' (vi) on the non-linear operators, if $\phi$
satisfies
\begin{equation}\label{thm_martingale_arretee_eq_1}
\,\, \rho_{\theta_k,\theta_{k+1}}[\phi (\theta_{k+1})]\leq  \phi( \theta_k) ({\rm resp.}=\phi( \theta_k) ),  \text{ for all }k, 
\end{equation} 
then,  for all $\tau\in\Theta$, we have
 \begin{equation}\label{eq22_prop_optional_sampling}
\rho_{\theta_k,\theta_{k+1}\wedge\tau}[\phi(\theta_{k+1}\wedge\tau)]\leq \phi(\theta_k\wedge\tau) ({\rm resp.}=\phi( \theta_k\wedge\tau) ).
\end{equation}

\end{theorem}
\begin{proof}
Let $k\in\N_0$.  
We have
\begin{equation}\label{eq_decomposition_prop_optional_sampling}
\rho_{\theta_k,\theta_{k+1}\wedge\tau}[\phi(\theta_{k+1}\wedge\tau)]=\I_{\{\tau\leq \theta_k\}}\rho_{\theta_k,\theta_{k+1}\wedge\tau}[\phi(\theta_{k+1}\wedge\tau)]+ \I_{\{\tau> \theta_{k}\}}\rho_{\theta_k,\theta_{k+1}\wedge\tau}[\phi(\theta_{k+1}\wedge\tau)].
\end{equation}
We note that on the set $\{\tau\leq \theta_k\}$, $\theta_{k+1}\wedge \tau = \theta_{k} \wedge \tau$. Hence, by the ``generalized zero-one law'', we have
$$\I_{\{\tau\leq \theta_k\}}\rho_{\theta_k,\theta_{k+1}\wedge\tau}[\phi(\theta_{k+1}\wedge\tau)] = \I_{\{\tau\leq \theta_k\}}\rho_{\theta_k,\theta_{k}\wedge\tau}[\phi(\theta_{k}\wedge\tau)].$$
As $\theta_{k} \wedge \tau \leq \theta_{k}$, by property (iii) of the non-linear evaluation $\rho$, we get
$$\rho_{\theta_k,\theta_{k}\wedge\tau}[\phi(\theta_{k}\wedge\tau)] = \phi(\theta_{k}\wedge\tau).$$
Hence, we have
\begin{equation}\label{eq_first_term_prop_optional_sampling}
\I_{\{\tau\leq \theta_k\}}\rho_{\theta_k,\theta_{k+1}\wedge\tau}[\phi(\theta_{k+1}\wedge\tau)] = \I_{\{\tau\leq \theta_k\}}\phi(\theta_{k}\wedge\tau).
\end{equation}
For the second term on the right-hand side of Equation \eqref{eq_decomposition_prop_optional_sampling}, we note that 
$\tau\wedge\theta_{k+1}=\theta_{k+1}$ on $\{\tau>\theta_k\}$. Hence, by the ``generalized zero-one law'' of the non-linear evaluation $\rho$, we have
\begin{equation*}
\begin{aligned}
\I_{\{\tau>\theta_k\}}\rho_{\theta_k,\theta_{k+1}\wedge\tau}[\phi(\theta_{k+1}\wedge\tau)]
&=\I_{\{\tau> \theta_k\}}\rho_{\theta_k,\theta_{k+1}}[\phi(\theta_{k+1})].
\end{aligned}
\end{equation*}
This, together with  Equation \eqref{thm_martingale_arretee_eq_1} on $\phi$ and the admissibility of $\phi$, gives 
\begin{equation}\label{eq_final_supp_prop_optional_sampling}
\I_{\{\tau> \theta_k\}}\rho_{\theta_k,\theta_{k+1}}[\phi(\theta_{k+1})]\leq \I_{\{\tau>\theta_k\}}\phi(\theta_{k})=\I_{\{\tau>\theta_k\}}\phi(\theta_{k}\wedge\tau).
\end{equation} 
By plugging in  \eqref{eq_first_term_prop_optional_sampling} and \eqref{eq_final_supp_prop_optional_sampling} in Equation \eqref{eq_decomposition_prop_optional_sampling},  we get 
$$ \rho_{\theta_k,\theta_{k+1}\wedge\tau}[\phi(\theta_{k+1}\wedge\tau)]\leq \I_{\{\tau\leq\theta_k\}}\phi(\theta_{k}\wedge\tau) +
\I_{\{\tau>\theta_k\}}\phi(\theta_{k}\wedge\tau) =\phi(\theta_{k}\wedge\tau).$$
This ends the proof. 
\end{proof}
We establish a characterization of $(\Theta,\rho)$-supermartingale (resp. $(\Theta,\rho)$-martingale)  families in the particular case where the sequence $(\theta_k)_{k\in\N_0}$ is constant from a certain term, \textit{independent of }$\omega$, onwards.  
\begin{proposition} \label{cor_optional_sampling} 
Assume that there exists $n\in\N_0$ such that $\theta_k=\theta_{n}=T$ a.s., for all $k\geq n$.  
Let $\phi=(\phi(\tau), \, \tau \in \Theta)$ be  a p-integrable admissible family. Under the assumptions of admissibility (ii), knowledge preservation (iii), monotonicity (iv), consistency (v) and ``generalized" zero-one law (vi) on the non-linear operators, if $  \,\, \rho_{\theta_k,\theta_{k+1}}[\phi (\theta_{k+1})]\leq  \phi( \theta_k) ({\rm resp.}=\phi( \theta_k) ),$ for all $k$,
 then, $\phi$
 is a $(\Theta, g)$-supermartingale \emph{(resp. $(\Theta, g)$--martingale)}  family.

\end{proposition}


\begin{proof}
We prove the result for the case of a $(\Theta,\rho)$-supermartingale family; the case of a $(\Theta,\rho)$-martingale family  can be treated similarly.
Let $\sigma, \tau$ in $\Theta$ be such that $\sigma\leq\tau$ a.s. 
As $\sigma\in\Theta$, we have  $\sigma=\sum_{k=0}^{n} \theta_k {\bf 1}_{A_k},$ where $(A_k)_{k\in\{0,\ldots, n\}}$ is a partition  of $\Omega$ such that $ A_k\in\cf_{\theta_k}$.    
We notice that in order to prove  $\rho_{\sigma,\tau}[\phi_{\tau}]\leq \phi_{\sigma}$, it is sufficient to prove the following property:
\begin{equation}\label{eq1_prop_optional_sampling}
\rho_{\theta_k\wedge\tau,\tau}[\phi(\tau)]\leq \phi(\theta_k\wedge\tau), \text{ for all } k\in\{0,1,\ldots, n\}.
\end{equation}
Indeed, this property proven, we will have
\begin{equation*}
\begin{aligned}
\rho_{\sigma,\tau}[\phi(\tau)]&=\rho_{\sigma\wedge\tau,\tau}[\phi(\tau)]=\sum_{k=0}^n \I_{A_k}\rho_{\theta_k\wedge\tau,\tau}[\phi(\tau)]\leq \sum_{k=0}^n \I_{A_k}\phi(\theta_k\wedge\tau)=\phi(\sigma\wedge\tau)=\phi(\sigma),
\end{aligned}
\end{equation*} 
where we have used the admissibility of $\rho$ to show the second equality. This will conclude the proof.
Let us now prove property \eqref{eq1_prop_optional_sampling}. We proceed by backward induction. 
For $k=n$, we have (recall that $\theta_n=T$)
$$\rho_{\theta_n\wedge\tau,\tau} [\phi(\tau)]=\rho_{T\wedge \tau,\tau}[\phi(\tau)]=\rho_{\tau,\tau}[\phi(\tau)]=\phi(\tau)=\phi(T\wedge\tau),$$
where we have used property (iii) to obtain the last but one equality.\\
[0.2cm]
We suppose that the property \eqref{eq1_prop_optional_sampling} holds true for $k+1.$
Then, by using this induction hypothesis, the time-consistency and  the monotonicity of the non-linear operators, 
we get
$$\rho_{\theta_k\wedge\tau,\tau}[\phi(\tau)]=\rho_{\theta_k\wedge\tau,\theta_{k+1}\wedge\tau}[\rho_{\theta_{k+1}\wedge\tau, \tau}[\phi(\tau)]]
\leq \rho_{\theta_k\wedge\tau,\theta_{k+1}\wedge\tau}[\phi(\theta_{k+1}\wedge\tau)].
$$
In order to conclude, it remains to prove 
\begin{equation}\label{eq2_prop_optional_sampling}
\rho_{\theta_k\wedge\tau,\theta_{k+1}\wedge\tau}[\phi(\theta_{k+1}\wedge\tau)]\leq \phi(\theta_k\wedge\tau).
\end{equation}
By Theorem \ref{thm_martingale_arretee}, we have 
\begin{equation}\label{eq_number_74}
  \I_{\{\tau\geq \theta_k\}}\rho_{\theta_k\wedge \tau,\theta_{k+1}\wedge\tau}[\phi(\theta_{k+1}\wedge\tau)]=\I_{\{\tau\geq \theta_k\}}\rho_{\theta_k,\theta_{k+1}\wedge\tau}[\phi(\theta_{k+1}\wedge\tau)]\leq 
\I_{\{\tau\geq \theta_k\}}\phi(\theta_k\wedge\tau).
\end{equation}
On the other hand, by the ``generalized zero-one law'', (applied with $A = \{\tau < \theta_{k}\}$ on which set, we have $\theta_{k+1} \wedge \tau = \tau$) we have
\begin{equation}\label{eq_number_75}
\I_{\{\tau<\theta_k\}}\rho_{ \theta_k\wedge\tau,\theta_{k+1}\wedge\tau}[\phi(\theta_{k+1}\wedge\tau)] = \I_{\{\tau<\theta_k\}}\rho_{ \tau, \tau}[\phi(\theta_{k+1}\wedge\tau)] = \phi(\tau)\I_{\{\tau<\theta_k\}},
\end{equation}
where we have used property (iii) on $\rho$ for the last equality.\\
[0.2cm]
Combining  Equations \eqref{eq_number_74} and \eqref{eq_number_75}, we deduce  \eqref{eq2_prop_optional_sampling}. 
The proposition is thus proved.
\end{proof}


\subsection{Dynamic programming principle in the case of finite number of pre-described stopping times}
We first introduce an explicit construction, by backward induction, of what will turn out to be the $(\Theta, \rho)$-Snell envelope of the pay-off family $\xi$, in this particular case of finite number of pre-described stopping times. Let $\xi$ be, as before, a $p$-integrable admissible family. \\
[0.2cm]
Let us  define the sequence of random variables $(U(\theta_{k}))_{k \in \{0, 1, ..., n\}}$ by backward induction as follows:
\begin{equation} \label{defSnell}
\left \{ \begin{array}{ll}
   U(\theta_{n}) := \xi(\theta_{n}), & \mbox{$k = n$};\\
   U(\theta_{k}) := \max (\xi(\theta_{k}); \rho_{\theta_{k}, \theta_{k+1}}[U(\theta_{k+1})]), & \mbox{for $k \in \{0, 1, ..., n-1\}$}. \end{array} \right. 
\end{equation}
From \eqref{defSnell} we see, by backward induction, that for each $k \in \{0, 1, ..., n\}$, $U(\theta_{k})$ is a well-defined real-valued random variable, which is $\mathcal{F}_{\theta_{k}}$-measurable and $p$-integrable. From \eqref{defSnell}, we also have $U(\theta_k)\geq \xi({\theta_k})$, for all $k\in \{0, 1, ..., n\}$.\\ Moreover, it can be shown that, for each $k \in \{0, 1, ..., n-1\}$,  $U(\theta_{k}) = U(\theta_{k+1})$ a.s. on $\{\theta_{k} = \theta_{k+1} \}$. Indeed, due to the second property (admissibility) and to the third property (knowledge preservation) of the non-linear operators $\rho$, we have
\begin{align*}
U(\theta_{k}) \mathbbm{1}_{\{\theta_{k} = \theta_{k+1}\}}& = \max(\xi(\theta_{k})\mathbbm{1}_{\{\theta_{k} = \theta_{k+1}\}}; \rho_{\theta_{k}, \theta_{k+1}}[U(\theta_{k+1})]\mathbbm{1}_{\{\theta_{k} = \theta_{k+1}\}})\\
& = \max(\xi(\theta_{k+1})\mathbbm{1}_{\{\theta_{k} = \theta_{k+1}\}}; \rho_{\theta_{k+1}, \theta_{k+1}}[U(\theta_{k+1})]\mathbbm{1}_{\{\theta_{k} = \theta_{k+1}\}}) \\
&= \max(\xi(\theta_{k+1})\mathbbm{1}_{\{\theta_{k} = \theta_{k+1}\}}; U(\theta_{k+1})\mathbbm{1}_{\{\theta_{k} = \theta_{k+1}\}}) = U(\theta_{k+1})\mathbbm{1}_{\{\theta_{k} = \theta_{k+1}\}},
\end{align*}
where, for the last equality, we have used that   $U(\theta_{k+1}) \geq \xi(\theta_{k+1})$.\\
[0.2cm]
Hence, 
$ U(\theta_{k}) = U(\theta_{k+1})$ a.s. on  $\{\theta_{k} = \theta_{k+1}\}.$
We can thus ``extend'' $U$ to the whole set $\Theta$ as follows. Let $\tau \in \Theta$. There exists a partition $(A_{k})_{k \in \{0, 1, ..., n\}}$ such that, for each $k$,  $A_{k} \in \mathcal{F}_{\theta_{k}}$, and such that $\tau = \sum^{n}_{k = 0} \theta_{k}\mathbbm{1}_{A_{k}}$. We set 
\begin{equation} \label{TEXequation3232}
U(\tau) := \sum^{n}_{k=0} U(\theta_{k}) \mathbbm{1}_{A_{k}}. 
\end{equation}
We show the following result:
\begin{theorem}\label{Thm_finite}
Under the assumptions of admissibility (ii), knowledge preservation  (iii), monotonicity (iv), consistency (v) and ``generalized zero-one law'' (vi) on the non-linear operators, the family $U := (U(\tau), \tau \in \Theta)$  defined by \eqref{defSnell} and \eqref{TEXequation3232} coincides with the $(\Theta, \rho)$-Snell envelope family of $\xi$.

\end{theorem}

For this, we first show an easy lemma, based on Proposition \ref{cor_optional_sampling}  and on the definition of the family $U$.

\begin{lemma}\label{Lemma_finite}
Under the assumptions of admissibility (ii), knowledge preservation (iii), monotonicity (iv), consistency (v) and ``generalized zero-one law'' (vi) on the non-linear operators, the family $U$ defined by \eqref{defSnell} and \eqref{TEXequation3232}
 is a $(\Theta,\rho)$-supermartingale family, dominating the family $\xi$. 

\end{lemma}

\begin{proof}
We already noticed, following the definition of $U(\theta_{k})$, that, for each $k$,  $U(\theta_k)\geq  \xi(\theta_k)$. Hence, for each $\tau\in \Theta$, $U(\tau)\geq \xi(\tau) $ a.s. (by definition of $U(\tau)$, cf. \eqref{TEXequation3232}).
The $(\Theta,\rho)$-supermartingale property of $U$ follows from the definition of $(U(\theta_{k}))_{k \in \{0, ..., n\}}$ and from  Proposition \ref{cor_optional_sampling}. 
\end{proof}
The following proof of Theorem \ref{Thm_finite} is a combination of Lemma \ref{Lemma_finite} and of  a proof of the minimality property of $U$. 
\begin{proof}[Proof of Theorem \ref{Thm_finite}]
By  Lemma \ref{Lemma_finite},   $U$ 
 is a $(\Theta,\rho)$-supermartingale family, dominating the family $\xi$. It remains to show that it is the minimal one.\\ 
Let  $\hat{U}$ be (another) $(\Theta, \rho)$-supermartingale family, such that $\hat{U}(\theta_{k}) \geq \xi(\theta_{k})$, for each $k \in \{0, 1, ..., n\}$.\\
At the terminal time  $\theta_n$, we have 
$\hat{U}(\theta_{n}) \geq \xi(\theta_{n}) = U(\theta_{n}).$\\
Let $k\in\{1,...,n\}. $ Suppose, by backward induction,  that $\hat{U}(\theta_{k}) \geq U(\theta_{k})$. We  need to show that $\hat{U}(\theta_{k-1}) \geq U(\theta_{k-1})$.\\
By the backward induction hypothesis and by the monotonicity of the non-linear operators $\rho_{\theta_{k-1}, \theta_{k}}$, we have 
$\rho_{\theta_{k-1}, \theta_{k}}[\hat{U}(\theta_{k})] \geq \rho_{\theta_{k-1}, \theta_{k}}[U(\theta_{k})].$
This, together with the definition of  $U(\theta_{k-1})$, gives  
$$U(\theta_{k-1}) = \max (\xi(\theta_{k-1}); \rho_{\theta_{k-1}, \theta_{k}}[U(\theta_{k})]) \leq \max (\xi(\theta_{k-1}); \rho_{\theta_{k-1}, \theta_{k}}[\hat{U}(\theta_{k})]).$$
Since $\hat{U}$ is a $(\Theta, \rho)$-supermartingale family, we have 
$\hat{U}(\theta_{k-1}) \geq \rho_{\theta_{k-1}, \theta_{k}}[\hat{U}(\theta_{k})].$
Hence, 
$$U(\theta_{k-1}) \leq \max(\xi(\theta_{k-1}); \rho_{\theta_{k-1}, \theta_{k}}[\hat{U}(\theta_{k})]) \leq \max(\xi(\theta_{k-1}); \hat{U}(\theta_{k-1})) = \hat{U}(\theta_{k-1}).$$
The reasoning by backward induction is thus finished and the minimality property of $U$ shown. We conclude that the family $U$ is equal to the smallest $(\Theta, \rho)$-supermartingale family dominating the family $\xi$, that is, to the $(\Theta, \rho)$-Snell envelope of the family $\xi$.
\end{proof}

\subsection{Optimal stopping times in the case of finite number of pre-described stopping times}
We define $\bar{\nu}_{k}$ by:
$$\bar{\nu}_{k} \coloneqq \essinf \mathcal{\bar{A}}_{k}, \; \; \; \text{where $\mathcal{\bar{A}}_{k} \coloneqq \{\tau \in \Theta_{\theta_{k}}: U(\tau) = \xi(\tau) \text{ a.s.} \}$.}$$
In the case of finite number of pre-described stopping times, we have:
\begin{align}\label{eq_35}
\bar{\nu}_{k} = \inf \{\theta_{l} \in \{\theta_{k}, ..., \theta_{n}\}: U(\theta_{l}) = \xi(\theta_{l})\}
= \min \{\theta_{l} \in \{\theta_{k}, ..., \theta_{n}\}: U(\theta_{l}) = \xi(\theta_{l})\}. 
\end{align}
Indeed, as the set $\{\theta_{l} \in \{\theta_{k}, ..., \theta_{n}\}: U(\theta_{l}) = \xi(\theta_{l})\}$ is a subset of $\mathcal{\bar{A}}_{k}$, we have:
\begin{align*}
\bar{\nu}_{k} = \essinf \{\tau \in \Theta_{\theta_{k}}: U(\tau) = \xi(\tau) \text{ a.s.} \} &\leq \essinf \{\theta_{l} \in \{\theta_{k}, ..., \theta_{n}\}: U(\theta_{l}) = \xi(\theta_{l}) \}\\
&= \inf \{\theta_{l} \in \{\theta_{k}, ..., \theta_{n}\}: U(\theta_{l}) = \xi(\theta_{l})\}.
\end{align*}
Let us now show the converse inequality:
$\bar{\nu}_{k} \geq \inf \{\theta_{l} \in \{\theta_{k}, ..., \theta_{n}\}: U(\theta_{l}) = \xi(\theta_{l}) \text{ a.s.}\}.$
As the set $\mathcal{\bar{A}}_{k}$ is stable by pairwise minimization, there exists a sequence $(\tau^{(m)})_{m \in \mathbbm{N}}$, such that: $(\tau^{(m)})$ is non-decreasing; for each $m$, $\tau^{(m)} \in \mathcal{\bar{A}}_{k}$; and $\lim_{m \to +\infty} \tau^{(m)} =\bar{\nu}_{k}$.\\
Let $\omega \in \Omega$ be given. For each $m \in \mathbbm{N}$, $\tau^{(m)}(\omega) \in \{\theta_{k}(\omega), \theta_{k+1}(\omega), ..., \theta_{n}(\omega) \}$, and moreover, $U(\tau^{(m)})(\omega) = \xi(\tau^{(m)})(\omega)$.\\
As all the elements of the sequence $(\tau^{(m)}(\omega))_{m \in \mathbbm{N}}$ are valued in $\{\theta_{k}(\omega), \theta_{k+1}(\omega), ..., \theta_{n}(\omega) \}$,  we have,
$\lim_{m \to +\infty} \tau^{(m)}(\omega) \in \{\theta_{k}(\omega), \theta_{k+1}(\omega), ..., \theta_{n}(\omega) \},$
which implies that from a certain rank onwards, the sequence is constant.\\ Moreover, we have
$U(\lim_{m \to +\infty} \tau^{(m)})(\omega) = \xi(\lim_{m \to +\infty} \tau^{(m)})(\omega).$
Thus, we have
$$\bar{\nu}_{k}(\omega) = \lim_{m \to +\infty} \tau^{(m)}(\omega) \in \{\theta_{k}(\omega), \theta_{k+1}(\omega), ..., \theta_{n}(\omega) \}, \text{ and } U(\bar\nu_{k})(\omega) = \xi(\bar\nu_{k})(\omega).$$
Hence, 
$\bar{\nu}_{k} \geq \essinf \{\theta_{l} \in \{\theta_{k}, ..., \theta_{n}\}: U(\theta_{l}) = \xi(\theta_{l}) \text{ a.s.}\}.$
As both inequalities hold true,  we conclude
\begin{equation}\label{TEXEQuaTioN383838}
\bar{\nu}_{k}= \inf \{\theta_{l} \in \{\theta_{k}, ..., \theta_{n}\}: U(\theta_{l}) = \xi(\theta_{l}) \text{ a.s.}\}.
\end{equation}

\begin{lemma} \label{TEXleMMa4242}
Under the assumptions of knowledge preservation (iii), consistency (v) and ``generalized zero-one law'' (vi) on $\rho$, the family $U=(U(\tau))$ is a $(\Theta, \rho)$-martingale on $[\theta_{k}, \bar{\nu}_{k}]$.

\end{lemma}

\begin{proof}
As $U$ satisfies the DPP (which it does by definition of $U$), we have, by Lemma \ref{TEXlemma3939}, for each $l \in \mathbbm{N}$,
\begin{equation}\label{TEXEQuatioNN5353}
U(\theta_{l} \wedge \bar{\nu}_{k}) = \rho_{\theta_{l} \wedge \bar{\nu}_{k}, \theta_{l+1} \wedge \bar{\nu}_{k}}[U(\theta_{l+1} \wedge \bar{\nu}_{k})].
\end{equation}
By Lemma \ref{Lemma2525}, to show that $U$ is a $(\Theta, \rho)$-martingale on $[\theta_{k}, \bar{\nu}_{k}]$, it is sufficient  to show that for any $\sigma$, such that $\theta_{k} \leq \sigma \leq \bar{\nu}_{k}$, it holds 
$U(\sigma) = \rho_{\sigma, \bar{\nu}_{k}}[U(\bar{\nu}_{k})].$
Let $\sigma \in \Theta_{\theta_{k}}$, such that $\sigma \leq \bar{\nu}_{k}$. Then, $\sigma = \sum_{i = k}^{n}\theta_{i}\mathbbm{1}_{A_{i}}$ and $\sigma \leq \bar{\nu}_{k}$. Thus, it is sufficient to show that for $i \in \{k, ..., n\}$, such that $\theta_{i} \leq \bar{\nu}_{k}$, it holds 
$\mathbbm{1}_{A_{i}}\rho_{\theta_{i}, \bar{\nu}_{k}}[U(\bar{\nu}_{k})] = \mathbbm{1}_{A_{i}}U(\theta_{i}),$
which is the same as, for each $i \in \{k, ..., n\}$, such that $\theta_{i} \leq \bar{\nu}_{k},$
$$\mathbbm{1}_{A_{i}}\rho_{\theta_{i}\wedge \bar{\nu}_{k}, \bar{\nu}_{k}}[U(\bar{\nu}_{k})] = \mathbbm{1}_{A_{i}}U(\theta_{i} \wedge \bar{\nu}_{k}).$$
We proceed by backward induction.
At rank $n$, we have
$$\rho_{\theta_{n}\wedge \bar{\nu}_{k}, \bar{\nu}_{k}}[U(\bar{\nu}_{k})] = \rho_{\bar{\nu}_{k}, \bar{\nu}_{k}}[U(\bar{\nu}_{k})] = U(\bar{\nu}_{k}) = U(\theta_{n}\wedge\bar{\nu}_{k}),$$
where we have used that $\bar{\nu}_{k} \leq \theta_{n} = T$, and the knowledge preserving property of $\rho$. \\
We suppose, by backward induction, that the property holds true at rank $i+1$. We show it at rank $i$.  By the consistency property and the backward induction hypothesis, we have
$$\rho_{\theta_{i}\wedge \bar{\nu}_{k}, \bar{\nu}_{k}}[U(\bar{\nu}_{k})] = \rho_{\theta_{i}\wedge \bar{\nu}_{k}, \theta_{i+1}\wedge\bar{\nu}_{k}}[\rho_{\theta_{i+1}\wedge\bar{\nu}_{k}, \bar{\nu}_{k}}[U(\bar{\nu}_{k})]] = \rho_{\theta_{i}\wedge \bar{\nu}_{k}, \theta_{i+1}\wedge\bar{\nu}_{k}}[U(\theta_{i+1}\wedge\bar{\nu}_{k})].$$
By Eq. \eqref{TEXEQuatioNN5353}, 
$\rho_{\theta_{i}\wedge \bar{\nu}_{k}, \theta_{i+1}\wedge\bar{\nu}_{k}}[U(\theta_{i+1}\wedge\bar{\nu}_{k})] = U(\theta_{i} \wedge \bar{\nu}_{k}).$
Hence, 
$\rho_{\theta_{i} \wedge\bar{\nu}_{k}, \bar{\nu}_{k}}[U(\bar{\nu}_{k})] = U(\theta_{i} \wedge \bar{\nu}_{k}),$
which completes the reasoning by backward induction.\\
We conclude that the family $U$ is a $(\Theta, \rho)$-martingale on $[\theta_{k}, \bar{\nu}_{k}]$.
\end{proof}
\noindent
We will now show that $U$ coincides with $V$, that $\bar{\nu}_{k}$ is optimal for the optimal stopping problem from time $\theta_{k}$-perspective, and that $\bar{\nu}_{k} = \nu_{k}$, where 
$${\nu}_{k} \coloneqq \essinf \mathcal{{A}}_{k}, \; \; \; \text{where $\mathcal{{A}}_{k} \coloneqq \{\tau \in \Theta_{\theta_{k}}: V(\tau) = \xi(\tau) \text{ a.s.} \}$.}$$

\noindent
 For this, we  do not need any type of  (Fatou)continuity assumption on $\rho$. 

\begin{theorem}
Under the assumptions of admissibility (ii), knowledge preservation (iii), monotonicity (iv), consistency (v), and  ``generalized zero-one law'' (vi) on the non-linear operators, we have\\
[0.2cm]
1. $U(\theta_{k}) = \rho_{\theta_{k}, \bar{\nu}_{k}}[\xi(\bar{\nu}_{k})] = V(\theta_{k}).$\\
[0.2cm]
2. $U = V$ and $\bar{\nu}_{k} = \nu_{k}$.

\end{theorem}

\begin{proof}
As $U \geq \xi$, and as $U$ is a $(\Theta, \rho)$-supermartingale (cf. Lemma \ref{Lemma_finite}), we have for any $\tau \in \Theta_{\theta_{k}}$,
$U(\theta_{k}) \geq \rho_{\theta_{k}, \tau}[U(\tau)] \geq \rho_{\theta_{k}, \tau}[\xi(\tau)],$
where we have used the monotonicity of $\rho_{\theta_{k}, \tau}$ for the second inequality.\\
Hence, 
\begin{equation} \label{TEXEQuationN5656}
U(\theta_{k}) \geq \esssup_{\tau \in \Theta_{\theta_{k}}}\rho_{\theta_{k}, \tau}[\xi(\tau)] = V(\theta_{k}).
\end{equation}
On the other hand, by Lemma \ref{TEXleMMa4242}, $U$ is a $(\Theta, \rho)$-martingale on $[\theta_{k}, \bar{\nu}_{k}]$. Moreover, by Eq. \eqref{TEXEQuaTioN383838}, we have $U(\bar{\nu}_{k}) = \xi(\bar{\nu}_{k})$. Hence,
$$U(\theta_{k}) = \rho_{\theta_{k}, \bar{\nu}_{k}}[U(\bar{\nu}_{k})] = \rho_{\theta_{k}, \bar{\nu}_{k}}[\xi(\bar{\nu}_{k})] \leq \esssup_{\tau \in \Theta_{\theta_{k}}}\rho_{\theta_{k}, \tau}[\xi(\tau)] = V(\theta_{k}).$$
We have thus showed:
$U(\theta_{k}) = V(\theta_{k}) = \rho_{\theta_{k}, \bar{\nu}_{k}}[\xi(\bar{\nu}_{k})],$
which proves statement 1 of the lemma.\\
Now, let us show statement 2. 
By admissibility of $U$ and $V$, it follows from statement 1, that, for any $\tau \in \Theta$, $U(\tau) = V(\tau)$.
Hence, $\bar{\nu}_{k} = \nu_{k}$ (from the definitions of $\bar{\nu}_{k}$ and $\nu_{k}$), which proves statement 2 of the lemma.
\end{proof}






\end{document}